\documentclass[proc]{edpsmath}
\usepackage{bbm}
\usepackage{float}
\usepackage{graphicx}                  
\usepackage{amssymb}
\usepackage{amsfonts}
\RequirePackage{amsmath}
\RequirePackage{amsthm}
\usepackage{mathrsfs} 
\usepackage{graphicx} 
\usepackage{color,subfigure} 
\usepackage{enumerate}  
\usepackage{ulem}

\newtheorem{theorem}{Theorem}
\newtheorem{lemma}{Lemma}

\newtheorem{remark}[theorem]{Remark}

\newcommand{\R}{\mathbb{R}}
\newcommand{\N}{\mathbb{N}}
\newcommand{\T}{\mathbb{T}}

\newcommand{\dt}{{\Delta t}}
\newcommand{\cM}{\mathcal{M}}
\newcommand{\cL}{\mathcal{L}}

\newcommand{\rme}{\mathrm{e}}
\newcommand{\Id}{\mathrm{Id}}
\newcommand{\ind}{\mathbf{1}}
\renewcommand{\leq}{\leqslant}
\renewcommand{\geq}{\geqslant}
\newcommand{\dps}{\displaystyle}
\newcommand{\eps}{\varepsilon}
\newcommand{\sig}{\sigma}
\newcommand{\ka}{\kappa}
\newcommand{\De}{\Delta}

\begin{document}


\title{Error analysis of the transport properties of Metropolized schemes}
\thanks{Funding from NEEDS ``Milieux poreux'' is gratefully acknowledged. We also benefited from the scientific environment of the Laboratoire International Associ\'e between the Centre National de la Recherche Scientifique and the University of Illinois at Urbana-Champaign.}

\thanks{We thank Marie Jardat and Vincent Dahirel for fruitful discussions on the practical aspects of the numerical methods we study here.}

\author{Max Fathi}\address{LPMA, 4 place Jussieu, 75005 Paris, France}
\author{Ahmed-Amine Homman}\address{CEA, DAM, DIF, F-91297 Arpajon, France}
\author{Gabriel Stoltz}\address{Universit\'e Paris-Est, CERMICS (ENPC), INRIA, F-77455 Marne-la-Vall\'ee, France}

\begin{abstract} We consider in this work the numerical computation of transport coefficients for Brownian dynamics. We investigate the discretization error arising when simulating the dynamics with the Smart MC algorithm (also known as Metropolis-adjusted Langevin algorithm). We prove that the error is of order one in the time step as $\dt$ goes to zero, when using either the Green-Kubo or the Einstein formula to estimate the transport coefficients. We illustrate our results with numerical simulations. \end{abstract}

\begin{resume} Nous nous int\'eressons dans cet article au calcul num\'erique des coefficients de transports pour des dynamiques browniennes. Nous \'etudions l'erreur de discr\'etisation qui apparait lorsqu'on simule la dynamique avec l'algorithme connu sous le nom de ``Smart MC'' dans la litt\'erature. Nous prouvons que cette erreur est d'ordre un en le pas de temps lorsque $\dt$ tend vers z\'ero, lorsqu'on utilise la formule de Green-Kubo ou la formule d'Einstein pour estimer les coefficients de transport. Nous illustrons ces r\'esultats avec des simulations num\'eriques. \end{resume}

\maketitle

Molecular simulation is nowadays a very common tool to quantitatively predict macroscopic properties of matter starting from a microscopic description. These macroscopic properties can be either static properties (such as the average pressure or energy in a system at fixed temperature and density), or transport properties (such as thermal conductivity or shear viscosity). Molecular simulation can be seen as the computational version of statistical physics, and is therefore often used by practitioners of the field as a black box to extract the desired macroscopic properties from some model of interparticle interactions. Most of the work in the physics and chemistry fields therefore focuses on improving the microscopic description, most notably developing force fields of increasing complexity. In comparison, less attention has been paid to the estimation of errors in the quantities actually computed by numerical simulation. Usually, due to the very high dimensionality of the systems under consideration, macroscopic properties are computed as ergodic averages over a very long trajectory of the system, evolved under some appropriate dynamics. There are two main types of errors in this approach: (i) statistical errors arising from incomplete sampling, and (ii) systematic errors (bias) arising from the fact that continuous dynamics are numerically integrated using a finite time-step~$\dt > 0$. 

The aim of this work is to understand the bias arising from the use of finite time steps in the computation of transport coefficients. We consider the case of the self-diffusion, for a certain type of dynamics called Brownian dynamics in the chemistry literature, discretized using the so-called ``Smart MC'' algorithm~\cite{RDF78,Jardat} (this algorithm was also rediscovered later on in the computational statistics literature~\cite{RR96}). The previous works on the numerical analysis of this dynamics established (i)~strong error estimates over finite times~\cite{BV09}, and, as a consequence, errors on finite time correlations~\cite{BE12}; (ii)~exponential convergence rates towards the invariant measure, uniformly in the timestep~\cite{BH13} (which holds up to a small error term in $\dt$ for systems in infinite volume).

\medskip

This proceedings is organized as follows. We start by describing in Section~\ref{sec:model} the Brownian dynamics and its discretization, and define the self-diffusion. We then provide in Section~\ref{sec:a_priori} a priori error estimates for the numerical estimation of the self-diffusion, through two different routes. Numerical simulations illustrate our error bounds in Section~\ref{sec:numerics}. We conclude in Section~\ref{sec:tracks} with some tracks to reduce the numerical error by appropriately modifying the numerical scheme. The proofs of our results are gathered in Section~\ref{sec:proofs}.

\section{Description of the model}
\label{sec:model}

\subsection{Brownian dynamics}

Consider $N$ particles with positions $q=(q_1,\dots,q_N)$ in a cubic box of size $L>0$: $q \in \mathcal{M} = (L\mathbb{T})^{dN}$, $\mathbb{T} = \mathbb{R} / \mathbb{Z}$ being the standard one-dimensional torus and $d$ being the physical dimension (usually $d=3$). The positions of the particles evolve according to the following dynamics:
\begin{equation}
\label{eq:dynamics}
dq_t = -\beta \nabla V(q_t) \, dt + \sqrt{2} \, dW_t,
\end{equation}
where $\beta = 1/(k_{\rm B}T)$ is the inverse temperature ($k_{\rm B}$ being Boltzmann's constant and $T$ being the temperature) and $W_t$ is a standard $dN$-dimensional Brownian motion. The function $V:\mathcal{M}\to \mathbb{R}$ is the potential energy, assumed to be smooth for the mathematical analysis. However, the numerical results presented in Section~\ref{sec:num_real_case} correspond to a potential energy function with singularities.

Standard results (see for instance the references in~\cite[Section~2.2]{LRS10}) show that~\eqref{eq:dynamics} admits the Boltzmann-Gibbs measure 
\begin{equation}
\label{eq:Gibbs}
\mu(dq) = Z^{-1} \,\rme^{-\beta V(q)} \, dq, \qquad Z = \int_\mathcal{M} \rme^{-\beta V},
\end{equation}
as its unique invariant probability measure (note that $Z$ is finite since the position space~$\cM$ is compact and~$V$ is smooth hence bounded). In fact, \eqref{eq:dynamics} is ergodic with respect to this measure, where ergodicity is understood both as (i) the long-time (almost-sure) convergence of averages along trajectories
\[
\lim_{t\to+\infty} \frac1t \int_0^t f(q_s) \, ds = \int_\cM f(q) \, \mu(dq) \qquad \mathrm{a.s.}
\]
for any initial condition~$q_0 \in \cM$ and all observables $f \in L^1(\mu)$; or as (ii) the convergence of the law $\psi(t,q) \, dq$ of the process~\eqref{eq:dynamics}, happening here at an exponential rate, for instance in total variation: Denoting with some abuse of notation the measure $\psi(t,q) \, dq$ by $\psi(t)$, there exist $C,\lambda > 0$ such that
\[
\| \psi(t) - \mu \|_{\rm TV} \leq C \, \rme^{-\lambda t},
\]
where the total variation distance between two measures $\nu_1,\nu_2$ is defined as
\[
\| \nu_1 - \nu_2 \|_{\rm TV} = 2 \sup_{S \in \mathscr{B}(\cM)} \left| \nu_1(S) - \nu_2(S)\right| = \sup_{|\varphi| \leq 1} \left| \int_\cM \varphi \, d\nu_1 - \int_\cM \varphi \, d\nu_2 \right|,
\]
the suprema being taken over all measurable sets of~$\cM$ for the first one, and over all bounded, measurable functions for the second one.

For further purposes, we introduce the generator of~\eqref{eq:dynamics}, namely the operator
\begin{equation}
\label{eq:generator}
\cL = -\beta \nabla V \cdot  \nabla + \Delta.
\end{equation}
This operator (defined with domain $D(\cL) = H^2(\mu)$) is self-adjoint on the Hilbert space $L^2(\mu)$ endowed with the scalar product 
\[
\left \langle \varphi, \psi \right\rangle_{L^2(\mu)} = \int_\cM \varphi \, \psi \, d\mu.
\] 
The operator $-\cL$ moreover has a positive spectral gap (see for instance~\cite[Section~2]{LRS10} and references therein). Indeed, a simple computation shows that
\begin{equation}
  \label{eq:quadratic_form}
  -\left \langle \cL \varphi, \varphi \right\rangle_{L^2(\mu)} = \frac1\beta \| \nabla \varphi \|^2_{L^2(\mu)}.
\end{equation}
The Poincar\'e inequality $\| \varphi \|_{L^2(\mu)} \leq C_{\cM,V} \| \nabla \varphi \|_{L^2(\mu)}$, valid for any function belonging to
\[
\widetilde{L}^2(\mu) = \left\{ \varphi \in L^2(\mu) \, \left| \int_\cM \varphi \, d\mu = 0 \right.\right\},
\] 
allows to conclude that 
\begin{equation}
\label{eq:spectral_gap}
\forall \varphi \in \widetilde{L}^2(\mu), \qquad -\left \langle \cL \varphi, \varphi \right\rangle_{L^2(\mu)} \geq \frac{1}{\beta C_{\cM,V}} \| \varphi \|_{L^2(\mu)}^2,
\end{equation}
which shows that the spectral gap is larger or equal to~$C_{\cM,V}^{-1}$. In particular, the resolvent $\cL^{-1}$ is a well-defined operator on~$\widetilde{L}^2(\mu)$, and the following estimate holds:
\begin{equation}
\label{eq:resolvent_estimate}
\left\| \cL^{-1} \right\|_{\mathcal{B}\left(\widetilde{L}^2(\mu)\right)} \leq \beta C_{\cM,V}.
\end{equation}
Here and in the sequel, for a given Banach space~$X$, we denote by $\mathcal{B}(X)$ the Banach space of bounded operators on~$X$, endowed with the norm 
\[
\| A \|_{\mathcal{B}(X)} = \sup_{x \in X \backslash\{0\}}\frac{\| Ax\|_X}{\|x\|_X}.
\]

\subsection{Self-diffusion}

The positions~$q_t$ are restricted to the periodic domain~$\cM$ and are therefore uniformly bounded in time. To obtain a diffusive behavior from the evolution of~$q_t$, we consider the following additive functional defined on the whole space~$\R^{d}$: starting from $Q_0 = q_0$,
\begin{equation}
  \label{eq:def_Q_t}
  Q_t = Q_0 - \beta \int_0^t \nabla V(q_s) \, ds + \sqrt{2} \, W_t.
\end{equation}
The difference with $q_t$ is that $Q_t$ is not reprojected in~$\cM$ by the periodization procedure (By this, we mean that we do not choose among all the images of $Q_t$ by translations of the lattice $L\mathbb{Z}^d$ the one for which all components are in the interval~$[0,L)$). The diffusion tensor is then given by the following limit (provided it exists): 
\begin{equation}
\label{eq:def_diffusion_coeff}
\mathscr{D} = \lim_{t \to +\infty} \mathbb{E}\left(\frac{Q_t-Q_0}{\sqrt{t}} \otimes \frac{Q_t-Q_0}{\sqrt{t}}\right),
\end{equation}
where the expectation is over all realizations of the continuous dynamics~\eqref{eq:dynamics}, starting from initial conditions distributed according to the Boltzmann-Gibbs measure~\eqref{eq:Gibbs}. The following result shows that the diffusion tensor~\eqref{eq:def_diffusion_coeff} is well defined, and naturally arises in a diffusive time-rescaling of the dynamics~\eqref{eq:dynamics}.

\begin{theorem}
\label{thm:cv_faible_processus}
Consider for $\varepsilon > 0$ the diffusively rescaled process $Q_t^\varepsilon = \varepsilon Q_{t/\varepsilon^2}$. Then, as $\varepsilon \to 0$, the process $Q_t^\varepsilon$ starting from a given initial condition $Q_0$ weakly converges on finite time intervals to an effective Brownian motion starting from~$Q_0$ and with covariance matrix $\mathscr{D}$ given by~\eqref{eq:def_diffusion_coeff}. Moreover, $\mathscr{D}$ is a real, positive definite $dN \times dN$ matrix, satisfying 
\[
0 < \mathscr{D} \leq 2\,\Id
\]
in the sense of symmetric matrices, and which can alternatively be expressed as
\begin{equation}
\label{eq:def_Einstein_continuous}
\mathscr{D} = 2\left( \mathrm{Id} - \beta^2 \int_0^{+\infty} \mathbb{E}\left[ \nabla V(q_t) \otimes \nabla V(q_0) \right] dt \right),
\end{equation}
where the expectation is over all realizations of the continuous dynamics~\eqref{eq:dynamics}, starting from initial conditions distributed according to the Boltzmann-Gibbs measure~\eqref{eq:Gibbs}.
\end{theorem}

The proof of this statement is standard, and follows from arguments presented in~\cite[Chapter~3]{BLP} for instance. We nonetheless provide a short proof in Section~\ref{sec:cv_faible_processus} since the proofs of the discrete counterparts of Theorem~\ref{thm:cv_faible_processus} rely on an appropriate extension of the argument used in the continuous case (see Section~\ref{sec:Einstein}).

A straightforward consequence of Theorem~\ref{thm:cv_faible_processus} is that the self-diffusion constant $\mathcal{D}$, defined as the average mean-square displacement of the individual particles, is well defined and has two equivalent expressions:
\begin{align}
\mathcal{D} = \frac{1}{2dN}\mathrm{Tr}(\mathscr{D})& = \lim_{t \to +\infty} \mathbb{E}\left(\frac{1}{2dNt} \sum_{i=1}^{N} (Q_{i,t} - Q_{i,0})^2 \right)  \label{eq:Einstein} \\
& = 1 - \frac{\beta^2}{dN} \int_0^{+\infty} \mathbb{E}\Big[ \nabla V(q_t)^T \nabla V(q_0)\Big] dt. \label{eq:GK}
\end{align}
The expression~\eqref{eq:Einstein} is called the \emph{Einstein formula}. The second expression~\eqref{eq:GK} involves an integrated autocorrelation function. In accordance with the standard physics and chemistry nomenclature, we call~\eqref{eq:GK} the \emph{Green-Kubo formula} for the self-diffusion in the sequel.

\subsection{Numerical estimation of the self-diffusion}

In order to compute approximations of formulas such as~\eqref{eq:Einstein} or~\eqref{eq:GK}, the first task is to numerically integrate realizations of the continuous dynamics~\eqref{eq:dynamics}. The most straightforward way would be to resort to a Euler-Maruyama scheme: given a time-step $\Delta t > 0$ and denoting by $q^n$ an approximation of $q_{n\Delta t}$, this scheme reads
\begin{equation}
\label{eq:EM}
q^{n+1} = q^n - \beta \dt\, \nabla V(q^n) + \sqrt{2\dt} \, G^n,
\end{equation}
where $(G^n)_{n \geq 0}$ is a sequence of independent and identically distributed (i.i.d.) $dN$-dimensional standard Gaussian random variables. However, this simple scheme has been shown to fail to be ergodic when the dynamics is considered on unbounded spaces and the potential energy function is not globally Lipschitz~\cite{MSH02}. In simulations of Brownian dynamics for ionic solutions, potential energy functions with Coulomb-type singularities are used and it has been observed that the energy blows up along trajectories of~\eqref{eq:EM} (see Section~\ref{sec:num_real_case}).

A way to stabilize the Euler-Maruyama scheme is to consider the configuration~\eqref{eq:EM} as a proposal move in a Metropolis-Hastings algorithm~\cite{MRRTT53,Hastings70}. This is precisely the Smart MC algorithm proposed in~\cite{RDF78} which was later called Metropolis adjusted Langevin algorithm in the computational statistics literature~\cite{RR96}. More precisely, starting from a configuration $q^n \in \cM$ (in fact seen as an element of~$\R^{dN}$), a new configuration $\widetilde{q}^{n+1} \in \R^{dN}$ is proposed according to~\eqref{eq:EM}, and then accepted with probability
\[
R_\dt\left(q^n,\widetilde{q}^{n+1}\right) = \min\left( \frac{\rme^{-\beta V(\widetilde{q}^{n+1})}T_\dt(\widetilde{q}^{n+1},q^n)}{\rme^{-\beta V(q^n)}T_\dt(q^n,\widetilde{q}^{n+1})},1\right),
\]
where
\[
T_\dt(q,q') = \left(\frac{1}{4\pi\dt}\right)^{dN/2} \exp \left (
-\frac{|q'-q+ \beta \dt \nabla V(q)|^2}{4\dt} \right )
\]
is the probability transition of the Markov chain~\eqref{eq:EM}. When the proposition is accepted, we project~$\widetilde{q}^{n+1}$ into the periodic simulation cell~$\cM$. If the proposal is rejected, the previous configuration is counted twice: $q^{n+1} = q^n$ (It is very important to count rejected configuration as many times as needed to ensure that the Boltzmann-Gibbs measure~$\mu$ is invariant). In conclusion,
\begin{equation}
\label{eq:mEM}
q^{n+1} = q^n + \ind_{U^n \leq R_\dt\left(q^n,\widetilde{q}^{n+1}\right)}\left(-\beta \dt \, \nabla V(q^n) + \sqrt{2\dt} \, G^n\right),
\end{equation}
where $U^n$ are i.i.d. uniform random variables in~$[0,1]$, and $\ind_{U^n \leq \alpha}$ is an indicator function whose value is~1 when $U^n \leq \alpha$ and~0 otherwise. The average rejection rate is 
\begin{equation}
\label{eq:avg_rejection_rate}
1 - R_\dt\left(q^n,\widetilde{q}^{n+1}\right) = 1 - \mathbb{E}_U\Big(\ind_{U \leq R_\dt\left(q^n,\widetilde{q}^{n+1}\right)}\Big).
\end{equation}
In the formula~\eqref{eq:mEM}, $q^{n+1}$ is considered as an element of the periodic box~$\cM$, while the proposed configuration~$\widetilde{q}^{n+1}$ is not reprojected into the simulation cell~$\cM$ and is therefore considered as an element of~$\R^{dN}$ (see the remark after~\eqref{eq:def_Q_t}). 

In order to avoid confusion, we call the scheme~\eqref{eq:mEM} ``Metropolized Euler-Maruyama'' in the sequel, and denote by $P_\dt$ its evolution operator:
\[
P_\dt f(q) = \mathbb{E}\Big( f\left(q^{n+1}\right) \, \Big| \, q^n = q \Big).
\]
By construction, the measure~\eqref{eq:Gibbs} is an invariant probability measure for this scheme, which is a reversible Markov chain. We refer to~\cite{BH13} for a study of the ergodic properties of the dynamics (in the more complicated case of dynamics on the full configuration space~$\mathbb{R}^{dN}$, subjected to a confining potential). 

Of course, the fact that some configurations are rejected destroys the trajectorial accuracy of the dynamics, see~\cite{BV09} for precise statements. The resulting strong errors and, as a consequence, errors on finite time correlation functions have been quantified in~\cite{BV09,BE12}, with prefactors which unfortunately depend on time. The estimates provided by these works therefore do not provide error estimates on diffusion coefficients, obtained either as infinite time integrals of correlation functions as in the Green-Kubo formula~\eqref{eq:GK} or as the infinite time average mean square displacement as in Einstein's formula~\eqref{eq:Einstein}.

The next section quantifies the errors in the approximation of~\eqref{eq:Einstein} and~\eqref{eq:GK} when the Metropolized Euler-Maruyama scheme is used. Although the formulas~\eqref{eq:Einstein} and~\eqref{eq:GK} are equivalent for continuous dynamics, they lead to different numerical methods. Let us already emphasize that the errors on the diffusion coefficients are in fact determined by the expansion of the evolution operator~$P_{\Delta t}$. This expansion is the same as the one used to establish weak error estimates. From a technical viewpoint, the techniques used in the proofs of our main results are therefore quite different from the techniques of~\cite{BV09,BE12}, which are based on strong error estimates obtained with Gronwall's lemma.

\section{A priori error estimates on the self-diffusion}
\label{sec:a_priori}

As discussed in~\cite{LMS13}, error bounds on transport properties in fact depend on approximation properties of the evolution operator (similar to the one used to prove weak error estimates), rather than strong error estimates. A key building block in this framework is the following expansion of the evolution operator, obtained by a slight extension of~\cite[Lemma~4.7]{BV09} and~\cite[Lemma~5.5]{BH13}.

\begin{lemma}
\label{lem:expansion_Pdt}
There exist an operator~$A$ and $\dt^* > 0$ such that, for any $0 < \dt \leq \dt^*$ and any smooth function~$\psi$, 
\begin{equation}
  \label{eq:expansion_P_dt}
  P_\dt \psi = \psi + \dt \, \cL \psi + \dt^2 A\psi + \dt^{5/2} r_{\psi,\dt}, 
\end{equation}
with a remainder $r_{\psi,\dt}$ uniformly bounded for $0 < \dt \leq \dt^*$. Moreover,
\begin{equation}
  \label{eq:vanishing_A}
  \int_\cM A \psi \, d\mu = 0.
\end{equation}
Finally, the average rejection~\eqref{eq:avg_rejection_rate} rate scales as $\dt^{3/2}$: There is a bounded function $\overline{\xi}$ such that, for any $p \in \mathbb{N}$, there exist~$C_p \geq 0$ and $\dt_p^* > 0$ for which
\begin{equation}
\label{eq:estimate_R_lemma}
\forall \, 0 < \dt \leq \dt_p^*, \qquad \mathbb{E}_G \left| R_\dt\left(q,q-\beta \dt\,\nabla V(q) + \sqrt{2\dt}\,G\right) \Big] - 1 + \dt^{3/2} \overline{\xi}(q) \right|^p \leq C_p \dt^{2p},
\end{equation}
where the expectation is over all possible realizations of the standard $dN$-dimensional Gaussian random variable~$G$. 
\end{lemma}

The precise expression of the operator~$A$ is unimportant. It is however given in the proof of this result, see Section~\ref{sec:proof_expansion_dt}. Note that the numerical scheme can be proved to be weakly first order accurate by relying on standard techniques~\cite{MT04}, in view of the equality 
\[
P_\dt\psi - \rme^{\dt \cL}\psi = \dt^2 \left(A - \frac12 \cL^2\right)\psi + \dt^{5/2} \widetilde{r}_{\psi,\dt}.
\]

Another important result which we will repeatedly use in the analysis below is the following uniform-in-$\dt$ geometric ergodicity of the Metropolized Euler-Maruyama scheme, easily obtained by adapting the results of~\cite{BH13} to the case of compact position spaces (for completeness, we nonetheless provide elements of proof in Section~\ref{sec:proof_lem_geom_ergod}). To state the result, we introduce the following functional space
\[
\widetilde{L}^\infty(\cM) = \left\{ f\in L^\infty(\cM) \, \left| \, \int_\cM f \, d\mu = 0 \right.\right\}.
\]

\begin{lemma}
\label{lem:geom_ergod}
There exists $\dt^* > 0$ and $C,\lambda > 0$ such that, for any $0 < \dt \leq \dt^*$, for all $n \in \N$ and any $f \in \widetilde{L}^{\infty}(\cM)$,
\begin{equation}
\label{eq:geom_ergod}
\left \| P_\dt^n f \right \|_{L^\infty} \leq C \, \rme^{-\lambda n \dt} \|f\|_{L^\infty}.
\end{equation}
As a consequence, there exists $K > 0$ such that
\begin{equation}
\label{eq:bound_discrete_generator}
\left\| \left(\frac{\Id - P_\dt}{\dt}\right)^{-1} \right\|_{\mathcal{B}(\widetilde{L}^{\infty})} \leq K.
\end{equation}
\end{lemma}

\subsection{Error estimates for the Green-Kubo formula}

We first give error estimates on~\eqref{eq:GK} by appropriately adapting the results from~\cite{LMS13}. The result is stated for two smooth observables $\psi,\varphi$ with average~0 with respect to~$\mu$. Define to this end
\[
\widetilde{C}^\infty(\mathcal{M}) = \left\{ \psi \in C^\infty(\mathcal{M}) \, \left| \int_\mathcal{M} \psi \, d\mu = 0 \right. \right\} \subset \widetilde{L}^{\infty}(\cM).
\] 
Error estimates for~\eqref{eq:GK} are obtained by setting $\psi = \varphi = \partial_{q_{i,\alpha}} V$, with $1 \leq i \leq N$ and $1 \leq \alpha \leq d$. Note indeed that a simple integration by parts shows that $\partial_{q_{i,\alpha}} V$ has average~0 with respect to~$\mu$, so $\partial_{q_{i,\alpha}} V \in \widetilde{C}^{\infty}(\cM)$.

\begin{theorem}
\label{thm:GK}
Consider two observables $\psi,\varphi \in \widetilde{C}^{\infty}(\cM)$, and define the modified observable
\[
\widetilde{\psi}_\dt = \left(\Id + \dt\, A\mathcal{L}^{-1}\right)\psi,
\]
where the operator~$A$ is defined in~\eqref{eq:expansion_P_dt}. Then, there exists $\dt^* > 0$ such that, for any $0 < \dt \leq \dt^*$,
\[
\int_0^{+\infty} \mathbb{E}\Big[\psi(q_t)\,\varphi(q_0)\Big]  dt = \dt \sum_{n=0}^{+\infty} \mathbb{E}_\dt \left[ \widetilde{\psi}_\dt(q^n)\, \varphi(q^0) \right] + \dt^{3/2} r_{\psi,\varphi,\dt},
\]
with $r_{\psi,\varphi,\dt}$ uniformly bounded (with respect to~$\dt$), and where the expectation on the left hand side of the above equation is with respect to initial conditions $q_0 \sim \mu$ and over all realizations of the dynamics~\eqref{eq:dynamics}, while the expectation on the right hand side is with respect to initial conditions $q^0 \sim \mu$ and over all realizations of the Metropolized Euler-Maruyama scheme~\eqref{eq:mEM}.
\end{theorem}

As a corollary, we obtain first order error bounds on the computation of the self-diffusion through~\eqref{eq:GK}:
\begin{equation}
  \label{eq:error_bounds_GK}
  \mathcal{D} = \mathcal{D}^{\rm GK}_\dt + \dt\, \widetilde{\mathcal{D}}^{{\rm GK},1} + \dt^{3/2}\, \widetilde{\mathcal{D}}^{\rm GK}_\dt,
\end{equation}
where $\widetilde{\mathcal{D}}^{\rm GK}_\dt$ is uniformly bounded for~$\dt$ sufficiently small, and where the numerically computed self-diffusion reads
\begin{equation}
\label{eq:approx_GK}
\mathcal{D}^{\rm GK}_\dt = \Id - \frac{\beta^2}{dN} \dt \sum_{n=0}^{+\infty} \mathbb{E}_\dt \left[ \nabla V(q^n)^T \nabla V(q^0) \right].
\end{equation}
The expression of the correction term is obtained by replacing the modified observable by its expression:
\[
\widetilde{\mathcal{D}}^{{\rm GK},1}_\dt = - \frac{\beta^2}{dN} \dt \sum_{n=0}^{+\infty} \mathbb{E}_\dt \left[ F(q^n)^T \nabla V(q^0) \right], 
\qquad 
F = A\cL^{-1} \nabla V.
\]
The appearance of subleading fractional correction term in~\eqref{eq:error_bounds_GK} (here, of order~$\dt^{3/2}$) is typical of Metropolis algorithms, and is usually not encountered for standard, un-Metropolized discretizations of SDEs (compare with the results of~\cite{LMS13}).

\subsection{Error bounds on the Einstein formula}

In this section, we investigate the discretization error made when using the Metropolized Euler-Maruyama scheme to approximate the self-diffusion using~\eqref{eq:Einstein}. In accordance with the definition~\eqref{eq:def_Q_t}, we introduce a discrete additive functional allowing to keep track of the diffuse behavior of the Markov chain: Starting from $Q^0 = q^0$,
\[
Q^n = \sum_{k=0}^{n-1} \delta_\dt\left(q^k,G^k,U^k\right),
\]
with
\begin{equation}
\label{eq:def_delta}
\delta_\dt\left(q^k,G^k,U^k\right) = \ind_{U^k \leq R_\dt\left(q^k,q^{k}-\beta \dt \, \nabla V(q^k) + \sqrt{2\dt} \, G^k\right)}\left(-\beta \dt \, \nabla V(q^k) + \sqrt{2\dt} \, G^k\right).
\end{equation}
While the Markov chain $(q^n)_{n \geq 0}$ defined by~\eqref{eq:mEM} remains in~$\cM$, the additive functional~$(Q^n)_{n \geq 0}$ has values in~$\R^{dN}$. The diffusion tensor actually computed by the numerical scheme is
\begin{equation}
\label{eq:def_D_Einstein_dt}
\mathscr{D}^{\rm Einstein}_\dt = \lim_{n \to +\infty} \mathbb{E}_\dt\left[\frac{Q^n - Q^0}{\sqrt{n\dt}} \otimes \frac{Q^n - Q^0}{\sqrt{n\dt}} \right], 
\end{equation}
where, as in Theorem~\ref{thm:GK}, the expectation on the right hand side is with respect to initial conditions $Q^0 = q^0 \sim \mu$ and for all realizations of the Metropolized Euler-Maruyama scheme.

\begin{theorem}
\label{thm:Einstein}
There exists $\dt^* > 0$ such that, for any $0 < \dt \leq \dt^*$, 
\[
\mathscr{D} = \mathscr{D}^{\rm Einstein}_\dt + \dt \, \widetilde{\mathscr{D}}_\dt,
\]
where the coefficients of the symmetric matrix $\widetilde{\mathscr{D}}_\dt \in \R^{dN \times dN}$ are uniformly bounded.
\end{theorem}

The proof of this result can be read in Section~\ref{sec:Einstein}. In fact, a slight extension of our technique of proof would allow to show that the diffusively rescaled process generated by the Metropolized Euler-Maruyama scheme, namely $\varepsilon Q^{\lfloor t/(\dt \, \varepsilon^2) \rfloor}$ (where $\lfloor x \rfloor$ denote the unique integer such that $\lfloor x \rfloor \leq x < \lfloor x \rfloor+1$), weakly converges on finite time intervals to a Brownian motion with covariance matrix~$\mathscr{D}_\dt$.

An immediate corollary of Theorem~\ref{thm:Einstein} is the following a priori error estimate on the self-diffusion:
\begin{equation}
\label{eq:error_bounds_Einstein}
\mathcal{D} = \mathcal{D}^{\rm Einstein}_\dt + \dt \, \widetilde{\mathcal{D}}^{\rm Einstein}_\dt, 
\qquad
\mathcal{D}^{\rm Einstein}_\dt = \frac{1}{2dN} \mathrm{Tr}\left(\mathscr{D}^{\rm Einstein}_\dt\right),
\end{equation}
where $\widetilde{\mathcal{D}}^{\rm Einstein}_\dt$ is uniformly bounded for $\dt$ sufficiently small. Some more work would allow to prove that the subleading correction term is of order~$\dt^{3/2}$, as in the Green-Kubo case (see Remark~\ref{rem:more_work}).

\section{Numerical illustration}
\label{sec:numerics}

The aim of this section is to illustrate the errors bounds~\eqref{eq:error_bounds_GK} and~\eqref{eq:error_bounds_Einstein}. We perform long computations so that the statistical errors are negligible. Let us mention that numerical simulations illustrating timestep errors for velocity autocorrelation functions were already presented in~\cite{BE12}.

\subsection{A simple one-dimensional case}

We start by considering a simple one-dimensional example ($N=d=1$): a single particle in the unit torus~$\cM = \mathbb{T}$, with the periodic potential $V(q) = \cos(2\pi q)$, at $\beta = 1$. Computations are performed by approximating expectations by realizations over $M$ replicas evolving independently, denoted by $q^{m,n}$ with $1 \leq m \leq M$ and where $n$ still is the step index. Initial conditions are prepared incrementally over the replicas. More precisely, starting from $q^{1,0} = 0$, we obtain the initial condition $q^{m+1,0}$ for the replica number $m+1$ by evolving the initial condition $q^{m,0}$ over 10~steps of the Metropolized Euler-Maruyama scheme with step size $\dt_{\rm thm} = 0.01$. We have checked that the equilibrium distribution is very well reproduced by the empirical measure produced by $\{ q^{m,0} \}_{1 \leq m \leq M}$ provided $M$ is reasonably large (say, $M \geq 10^3$).

The self-diffusion coefficient $\mathcal{D}_\dt^{\rm Einstein}$ for the Einstein approach is approximated by fitting the unnormalized self-diffusion
\begin{equation}\label{eq:estimateur_diffusion}
D^M_n = \frac1M \sum_{m=1}^M \left( Q^{m,n} - Q^{m,0} \right)^2
\end{equation}
by a linear function $\mathcal{D}^{{\rm Einstein},M}_\dt \, n\dt$, the slope being the estimation of the self-diffusion for the time step under consideration. This is indeed confirmed by Figure~\ref{fig:diffusion} (Left), which presents the evolution of $D_n^M$ as a function of the physical time $n\dt$. The results produced in Figures~\ref{fig:diffusion} and~\ref{fig:error_diffusion} have been obtained with $M = 10^7$ replicas and $n_{\rm Einstein} = 3 \times 10^5$ steps. Let us also note that, in accordance with~\eqref{eq:estimate_R_lemma}, the rejection rate scales as $\dt^{3/2}$. 

A numerical approximation of~\eqref{eq:approx_GK} requires both a discretization using finitely many replicas, but also a truncation of the integration in time with an upper bound~$\tau$. We consider the following numerical estimation of the self-diffusion coefficient obtained with the Green-Kubo formula:  
\begin{equation}\label{eq:estimateur_GK}
\mathcal{D}^{{\rm GK},M,\tau}_\dt = 1 - \frac{\beta^2 \dt}{M} \sum_{m=1}^M \sum_{n=0}^{\lfloor \tau/\dt \rfloor} V'(q^{m,n}) V'(q^{m,0}).
\end{equation}
The correlation functions shown in Figure~\ref{fig:correlation} suggest that the autocorrelation of $V'$ is exponentially decreasing. It can be considered as negligible for times larger than~0.2. The numerical results reported in Figure~\ref{fig:correlation} and~\ref{fig:error_diffusion} have been obtained with $M = 2 \times 10^8$ replicas and a time cut-off $\tau = 0.3$.

The results presented in Figure~\ref{fig:error_diffusion} indeed confirm that, for small time steps~$\dt$, the error in the self-diffusion is of order~$\dt$ for both methods. For larger time steps, nonlinearities appear. Note also that the errors on the coefficients computed with the Green-Kubo formula are smaller in this simple case. In any case, in accordance with the statements of Theorem~\ref{thm:cv_faible_processus}, the self-diffusion is between~0 and~1 when $\dt \to 0$.

\begin{figure}
\begin{center}
\includegraphics[width=7cm]{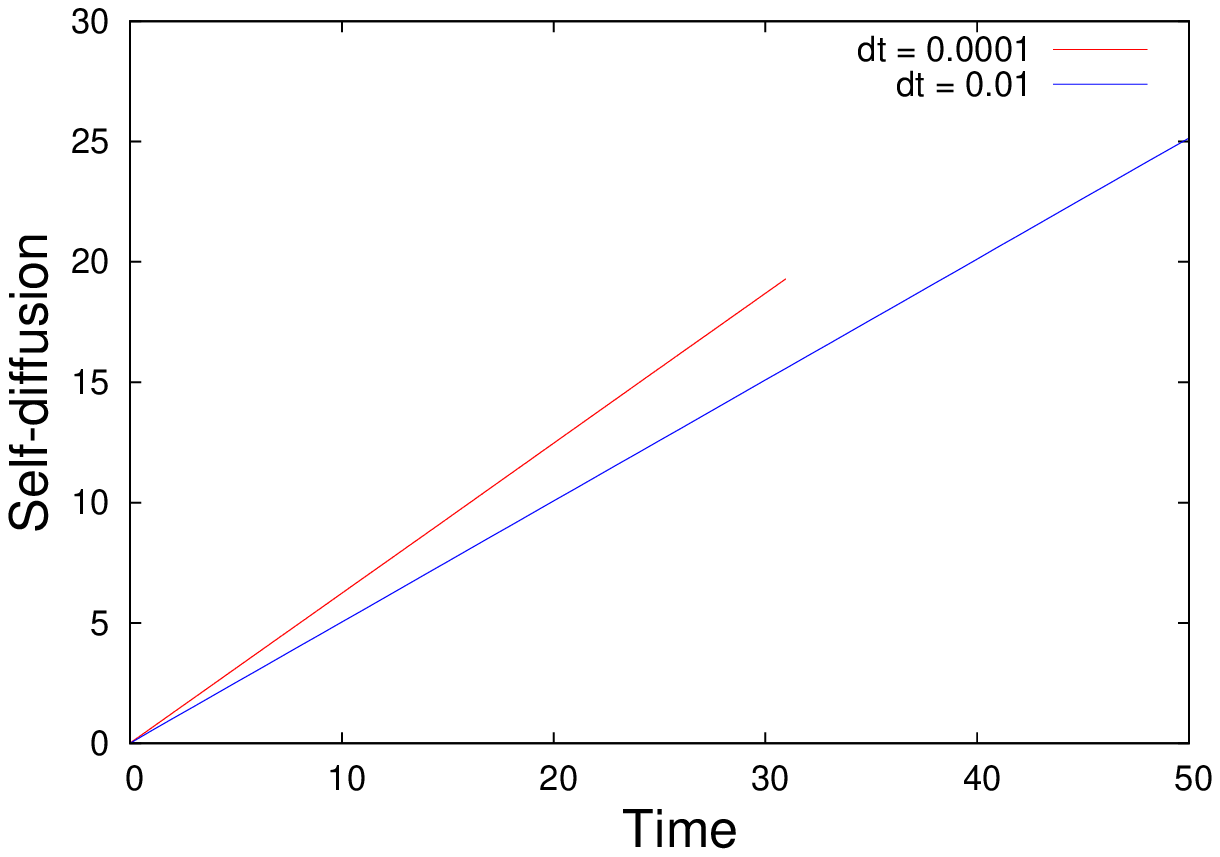}
\includegraphics[width=7cm]{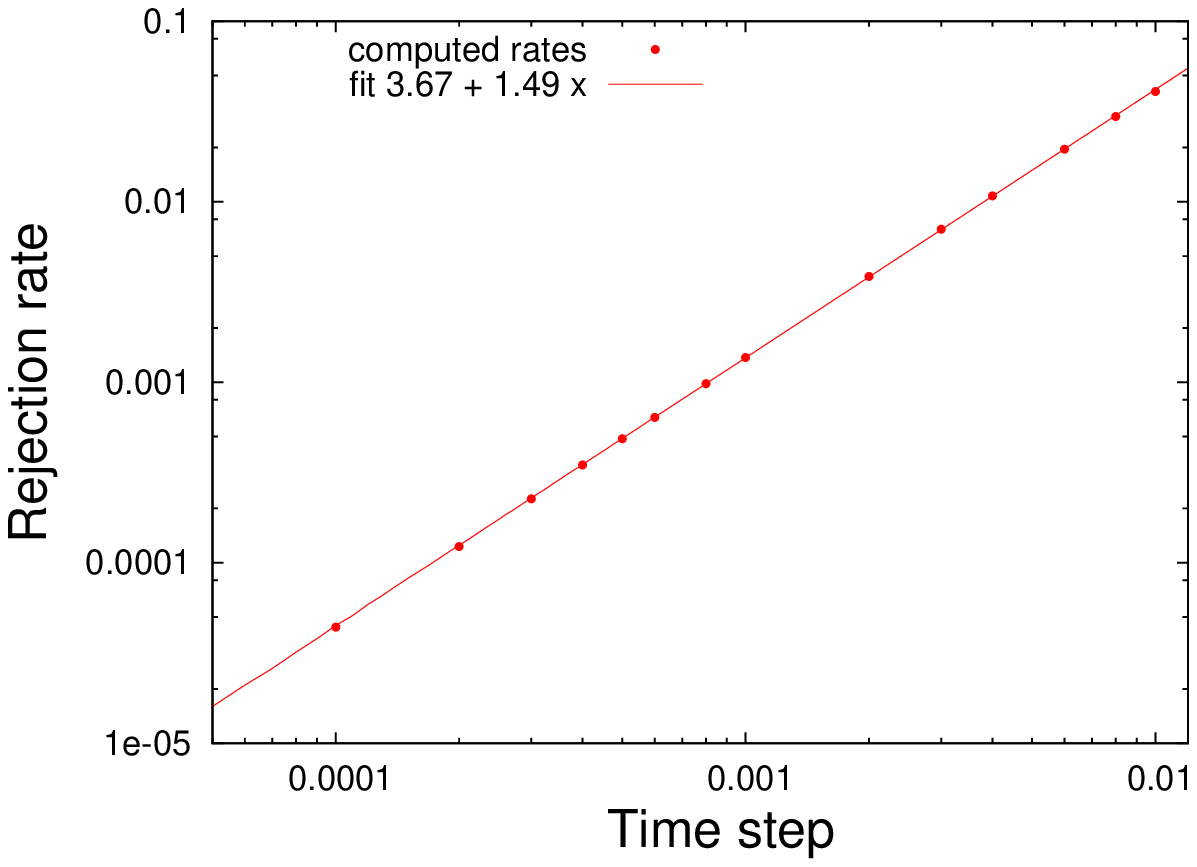}
\end{center}
\caption{\label{fig:diffusion} Left: Self-diffusion $D_n^M$ as a function of the physical time $n\dt$ for two values of the time step~$\dt$. Right: average rejection rate~\eqref{eq:avg_rejection_rate} as a function of the time step~$\dt$, in a log-log scale. As predicted by~\eqref{eq:estimate_R_lemma}, the rejection rate scales as $\dt^{3/2}$.}
\end{figure}

\begin{figure}
\begin{center}
\includegraphics[width=7cm]{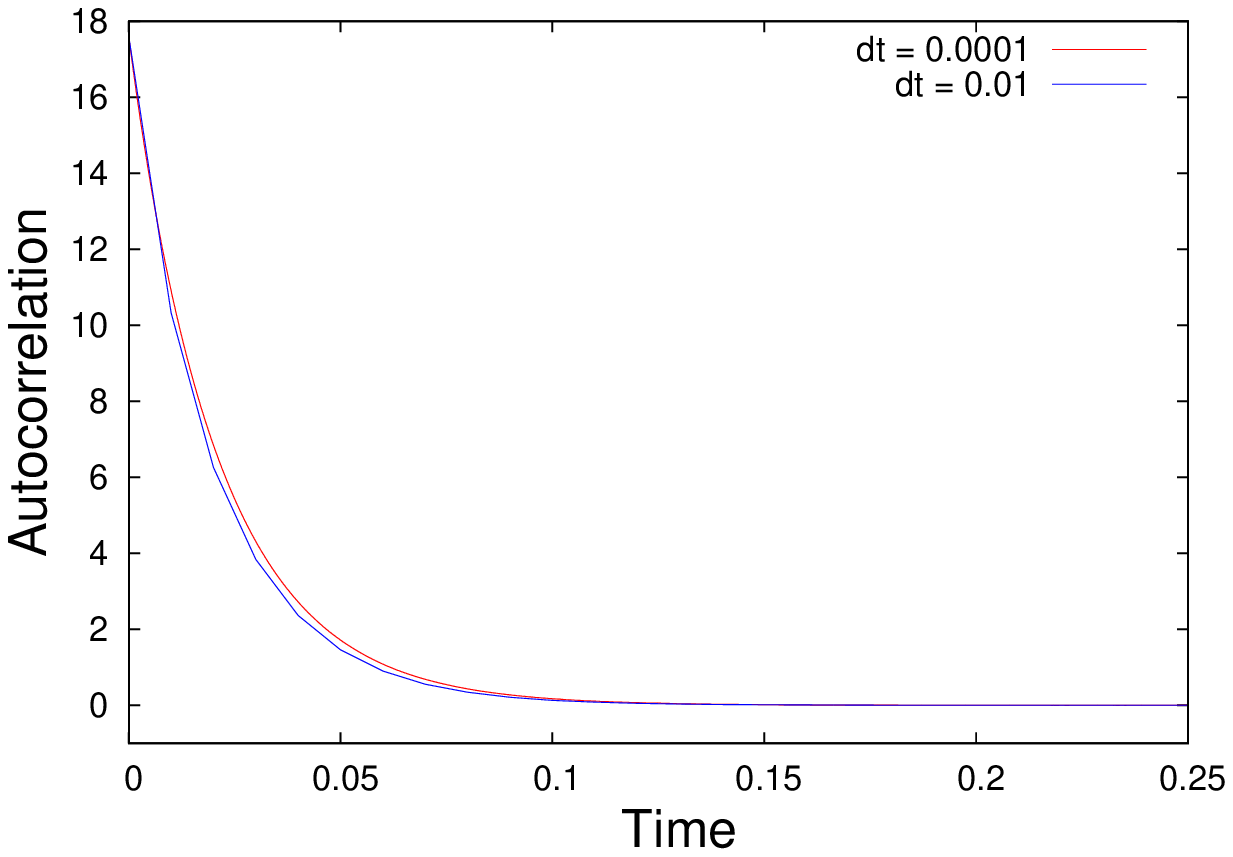}
\includegraphics[width=7cm]{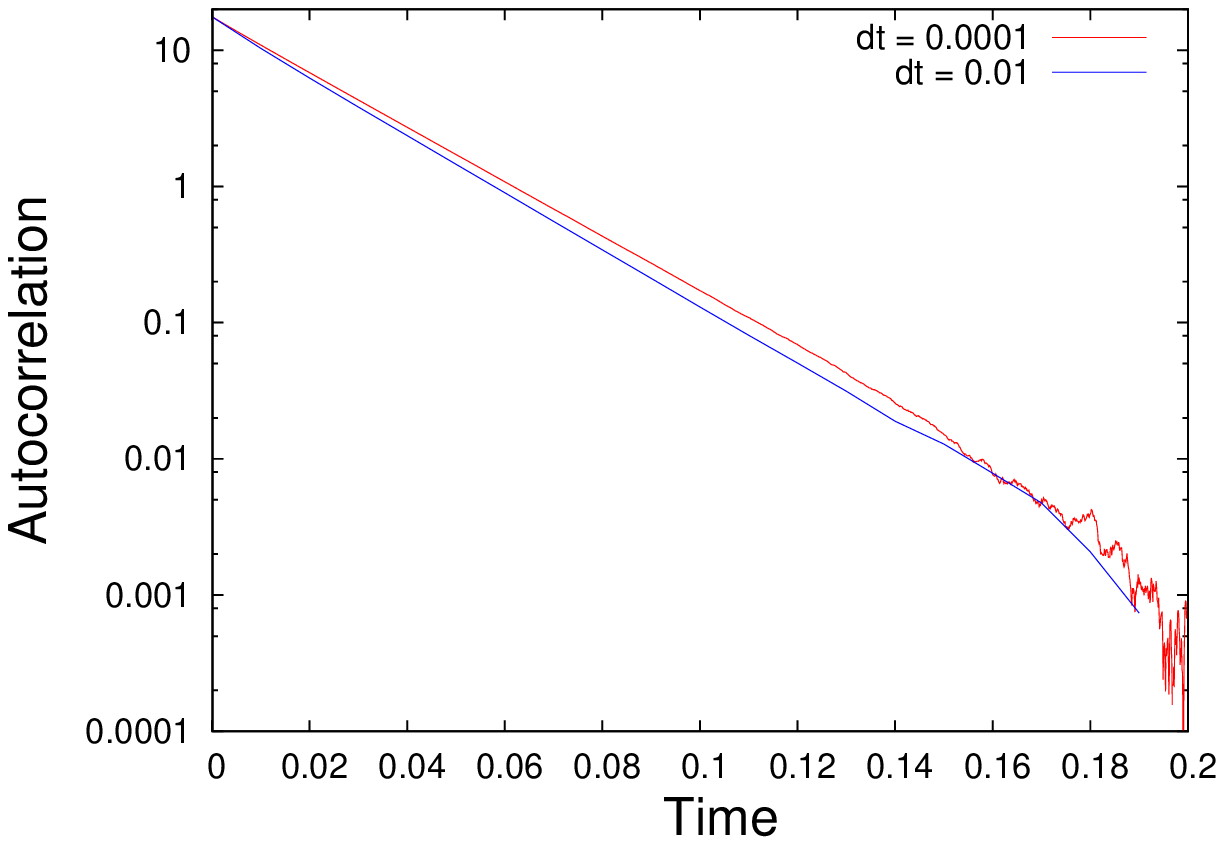}
\end{center}
\caption{\label{fig:correlation} Plot of the approximated correlation functions $\mathbb{E}(V'(q_t) V'(q_0))$. Left: standard view. Right: logarithmic scale on the ordinates.}
\end{figure}

\begin{figure}
\begin{center}
\includegraphics[width=7cm]{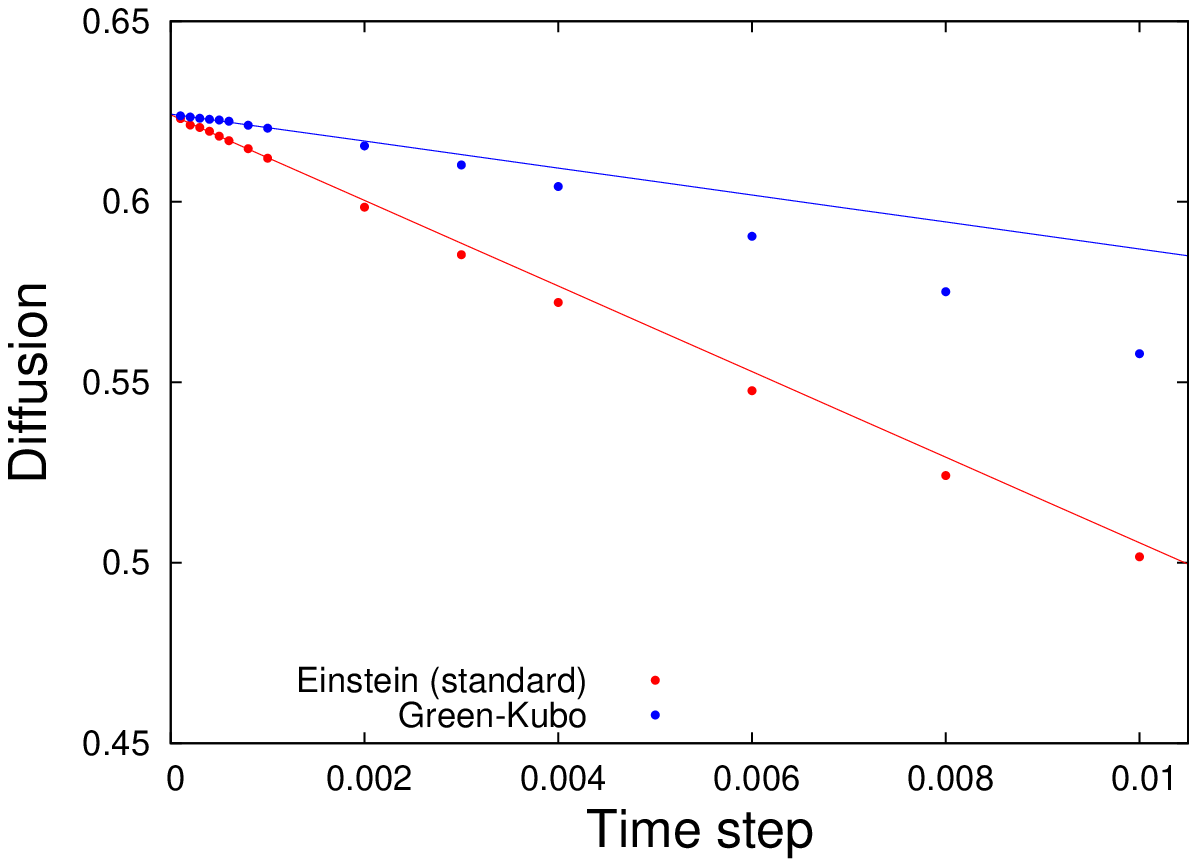}
\includegraphics[width=7cm]{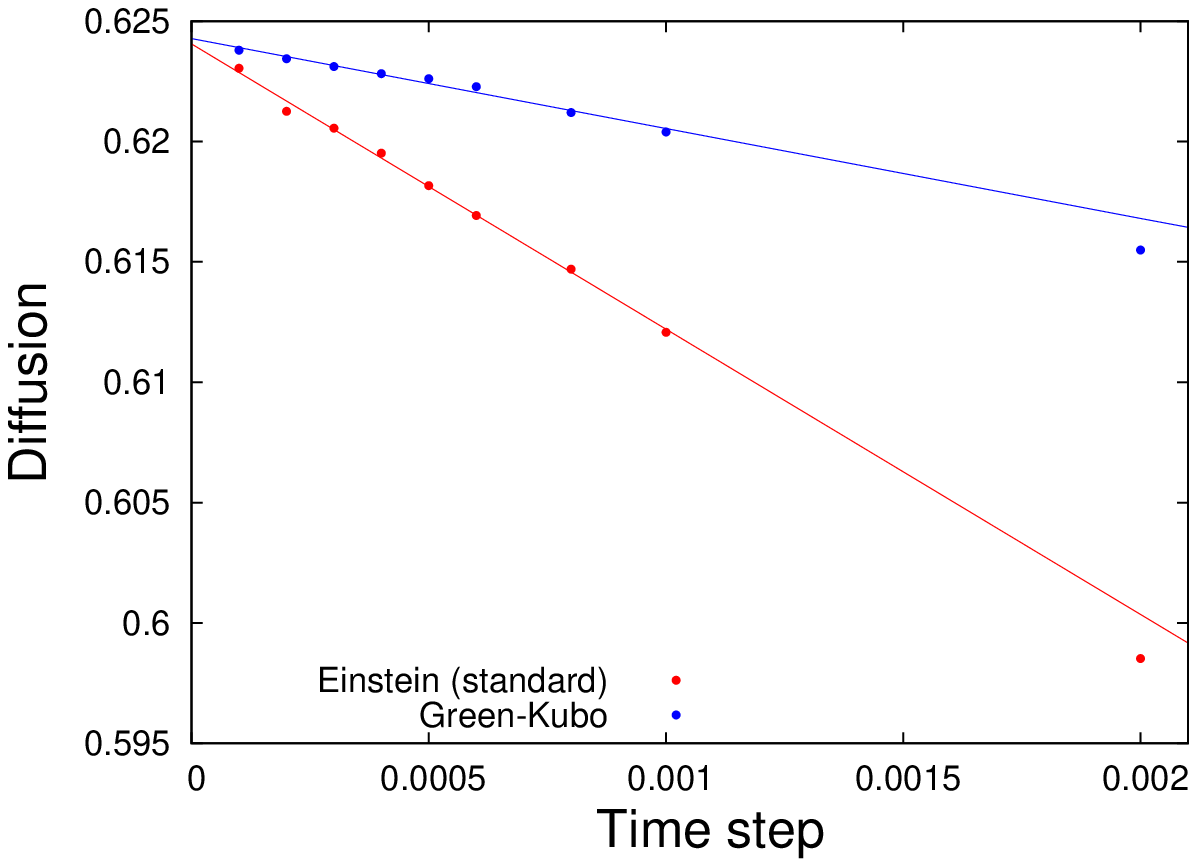}
\end{center}
\caption{\label{fig:error_diffusion} Diffusion constant as a function of the time step $\dt$ for the one-dimensional potential $V(q) = \cos(2\pi q)$ at $\beta = 1$, with a zoom on the smaller time steps on the right picture. Affine fits $\mathcal{D}_{\Delta} = \mathcal{D} + \Delta t \, \widetilde{D}_1$ consistent with~\eqref{eq:error_bounds_GK} and~\eqref{eq:error_bounds_Einstein} are in both cases represented by solid lines.}
\end{figure}

\subsection{The more realistic case of solvated ions}
\label{sec:num_real_case}

We consider in this section a more physically realistic system: a large, fixed ion interacting with smaller particles, typically smaller ions. The $N$ smaller particles evolve in a three-dimensional cubic simulation box of length~$L$ with periodic boundary conditions. The smaller particles have the same mass~$m$, and their positions are denoted by $q_i \in (L\mathbb{T})^3$. The potential energy functions are inspired by standard choices in the modeling of ionic solutions~\cite{Jardat_thesis}. The interaction between small particles is governed by an appropriately truncated Lennard-Jones potential:
\begin{equation}
\label{LJ pot}
v(r) = 4\eps\left( \left(\frac{\sig}r\right)^{12} - \left(\frac{\sig}r\right)^6 \right) - \eps_{\mathrm{shift}} - f_{\mathrm{spline}}(r-r_{\rm cut}), \qquad \mathrm{when} \ r \leq r_{\rm cut},
\end{equation}
and $v(r) = 0$ for $r \geq r_{\rm cut}$. The parameter $\eps > 0$ is some reference energy, while $\sigma > 0$ is some reference distance. The parameters $\eps_{\mathrm{shift}}$ and $f_{\mathrm{spline}}$ ensure that $v$ is a $C^1$ function. When $r_{\rm cut} \to +\infty$, the minimal energy of $v$ converges to~$-\eps$, a value obtained at a distance $r_{\rm min} = 2^{1/6} \sigma$.

We additionally consider a large ion, modeled as a fixed particle at position $q_{\rm ion}$ (the center of the simulation box), whose interaction with the solvent particles is described by an attractive Yukawa potential (screened Coulomb interaction) plus some repulsive potential preventing the small particles from coming too close to the ion. More precisely, for a solvent particle at position~$q$, the interaction reads $v_{\rm ion}(|q-q_{\rm ion}|)$, with
\begin{equation}
\label{pot ion-part}
v_{\rm ion}(r) = E_{\mathrm{min}} \left(1-\frac{1+\ka\sig}{24}\right)^{-1}\left( \frac {1+\ka\sig} {24}\left(\frac{\sigma}{r}\right)^{24} - \frac{\sig}{r} \rme^{-\kappa (r-\sig)} \right)- \eps_{\mathrm{shift}}^{\rm ion} - f_{\mathrm{spline}}^{\rm ion}(r-r_{\mathrm{cut}}^{\rm ion}), \qquad r \leq r_{\rm cut}^{\rm ion},
\end{equation}
and $v_{\rm ion}(r) = 0$ for $r \geq r_{\rm cut}^{\rm ion}$. The parameters $\eps_{\mathrm{shift}}^{\rm ion}$ and $f_{\mathrm{spline}}^{\rm ion}$ ensure as above that $v_{\rm ion}$ is~$C^1$. The energy $-E_{\mathrm{min}} \leq 0$ is the minimal value of the potential, obtained when $r=\sig$ (in the limit when $r_{\rm cut}^{\rm ion} \to +\infty$ and $f_{\rm shift}^{\rm ion} = 0$) while $\kappa$ is some inverse length. The total potential energy of the $N$ small particles finally reads
\[
V(q_1,\dots,q_N) = \sum_{1 \leq i < j \leq N} v(|q_i-q_j|) + \sum_{i=1}^N v_{\rm ion}(|q_i-q_{\rm ion}|).
\]
The potentials~$v$ and~$v_{\rm ion}$ are plotted in Figure~\ref{fig:potential_LJ}.

\begin{figure}
\begin{center}
\includegraphics[width=7cm]{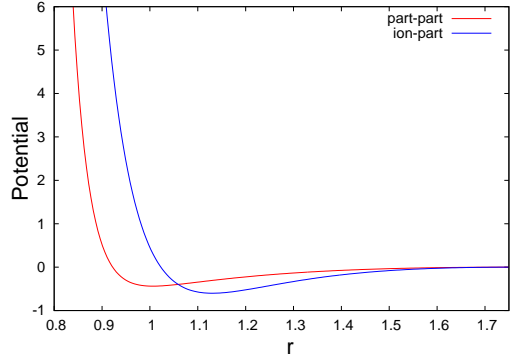}
\end{center}
\caption{\label{fig:potential_LJ} Plot of the particle-particle interaction~$v$ (red), and of the ion-particle interaction~$v_{\rm ion}$ (blue).}
\end{figure}

The results of this section are expressed in the reduced units obtained from the Lennard-Jones energy~$\eps$, the Lennard-Jones distance~$\sigma$ and the mass~$m$. In particular, the reference time is $t^* = \sigma \sqrt{m/\eps}$. Simulations were performed using the following parameters: $N=20$, solvent density $\rho = N/L^3 = 0.4$, $E_{\mathrm{min}} = 0.8347$, $\kappa = 1.7025$, inverse temperature $\beta = 1$, and $r_{\mathrm{cut}} = r_{\rm cut}^{\rm ion} = 1.76$. For this choice of parameters, we have observed that the simulations blow up for time-steps $\dt$ of the order of~$4\times10^{-4}$ when using the Euler-Maruyama (un-Metropolized) scheme~\eqref{eq:EM}; whereas the Metropolized Euler scheme~\eqref{eq:mEM} allows for much larger time steps. 

Expectations are approximated using $M$ trajectories of the system. We integrate trajectories one after the other, using the Metropolized Euler scheme~\eqref{eq:mEM}, with initial conditions for the $(m+1)$th trajectory obtained by taking the last configuration of the $m$th trajectory. The self-diffusion coefficient calculated with the Green-Kubo formula is obtained as in the previous section, using the estimator~\eqref{eq:estimateur_GK} (upon introducing the correct normalization factor $1/(3N)$ in the autocorrelation and replacing $V'$ by $\nabla V$). The values obtained by the Einstein formula are computed by dividing the unnormalized diffusion $D_n^M$ defined in~\eqref{eq:estimateur_diffusion} by the final time of the simulation: for a simulation time~$\tau$, 
\begin{equation}\label{eq:diffusion_coeff_Einstein}
D^{\mathrm{Einstein},M,\tau}_{\De t} = \frac1{6N\tau} D^M_{\lfloor \tau/\dt \rfloor}.
\end{equation}
The results presented in Figure \ref{fig:diffusion_LJ} show that the unnormalized mean squared displacement indeed grows linearly in time, as expected. Note also that, in accordance with~\eqref{eq:estimate_R_lemma}, the rejection rate scales as $\dt^{3/2}$. The results of Figure \ref{fig:diffusion_LJ} and \eqref{fig:error_diffusion_LJ} have been obtained with $M=10^5$ trajectories, with integrations performed up to $\tau=20$.

The results presented in Figure~\ref{fig:correlation_LJ} suggest that the decay of the force autocorrelation cannot be represented by a single exponential function. The force autocorrelation can be considered to be small in relative value for times of the order of~$0.1$. The numerical results reported in Figure~\ref{fig:correlation_LJ} and~\ref{fig:error_diffusion_LJ} were obtained by averaging $M=10^6$ trajectories with an integration time~$\tau=0.3$.

Error estimates for the diffusion coefficients are gathered in Figure~\ref{fig:error_diffusion_LJ}. For small time steps $\De t$, the error in the self-diffusion is linear in~$\De t$, while nonlinearities appear for larger time steps. The continuous lines are linear fits obtained over the values corresponding to the 10 smallest time steps. As in the simple example discussed in the previous section, estimates obtained with Green-Kubo's formula seem more reliable than those obtained with Einstein's formula. Note also that, in accordance with Theorem \ref{thm:cv_faible_processus}, the self-diffusion is between 0 and 1 in all cases for sufficiently small time-steps.

\begin{figure}
\begin{center}
\includegraphics[width=7cm]{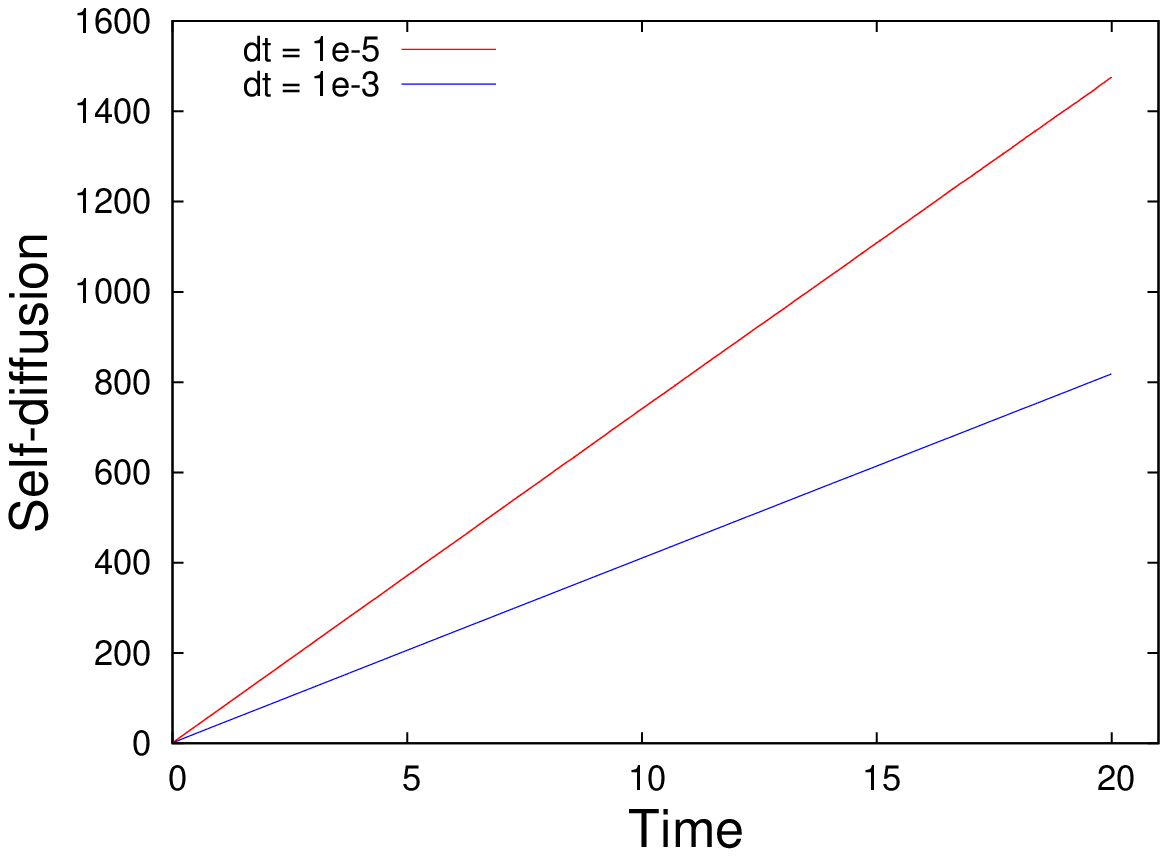}
\includegraphics[width=7cm]{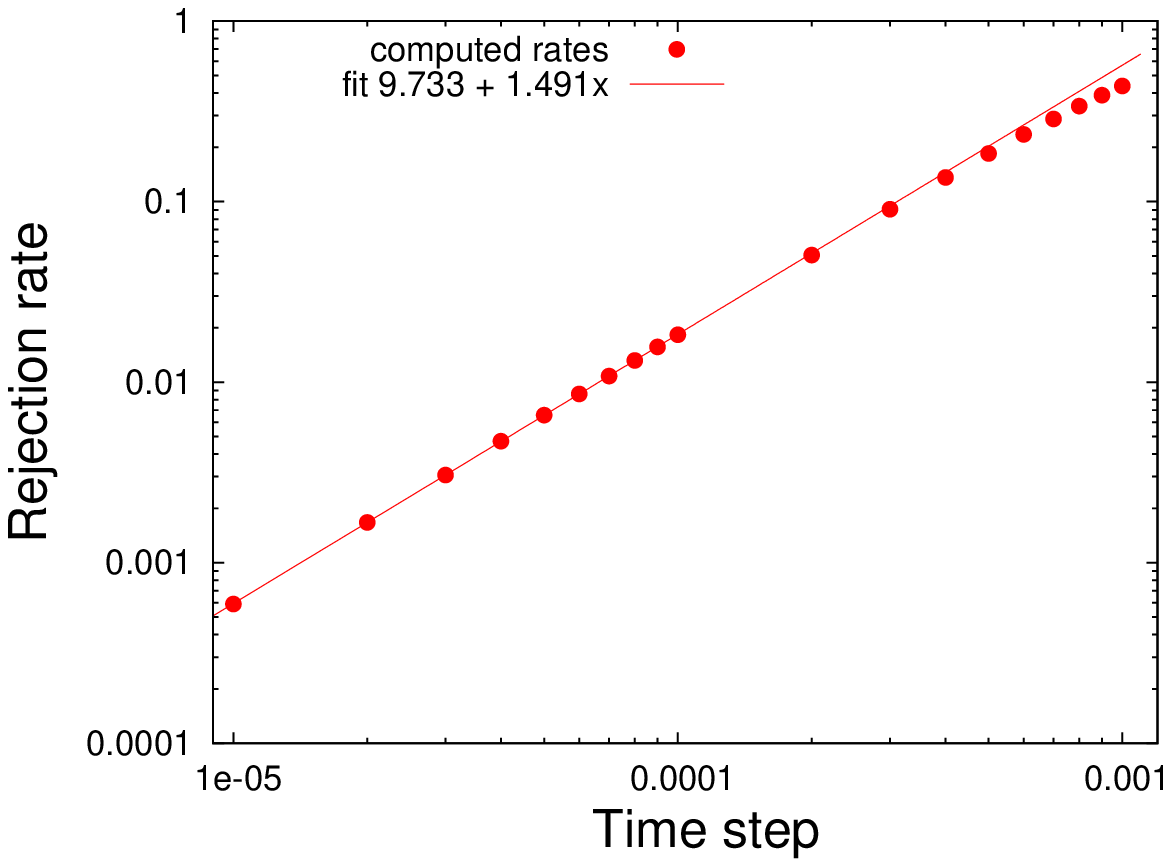}
\end{center}
\caption{\label{fig:diffusion_LJ} Left: Self-diffusion $D_n^M$ as a function of the physical time $n\dt$ for two values of the time step~$\dt$. Right: average rejection as a function of the time step~$\dt$, in a log-log scale. As predicted by~\eqref{eq:estimate_R_lemma}, the rejection rate scales as $\dt^{3/2}$.}
\end{figure}

\begin{figure}
\begin{center}
\includegraphics[width=7cm]{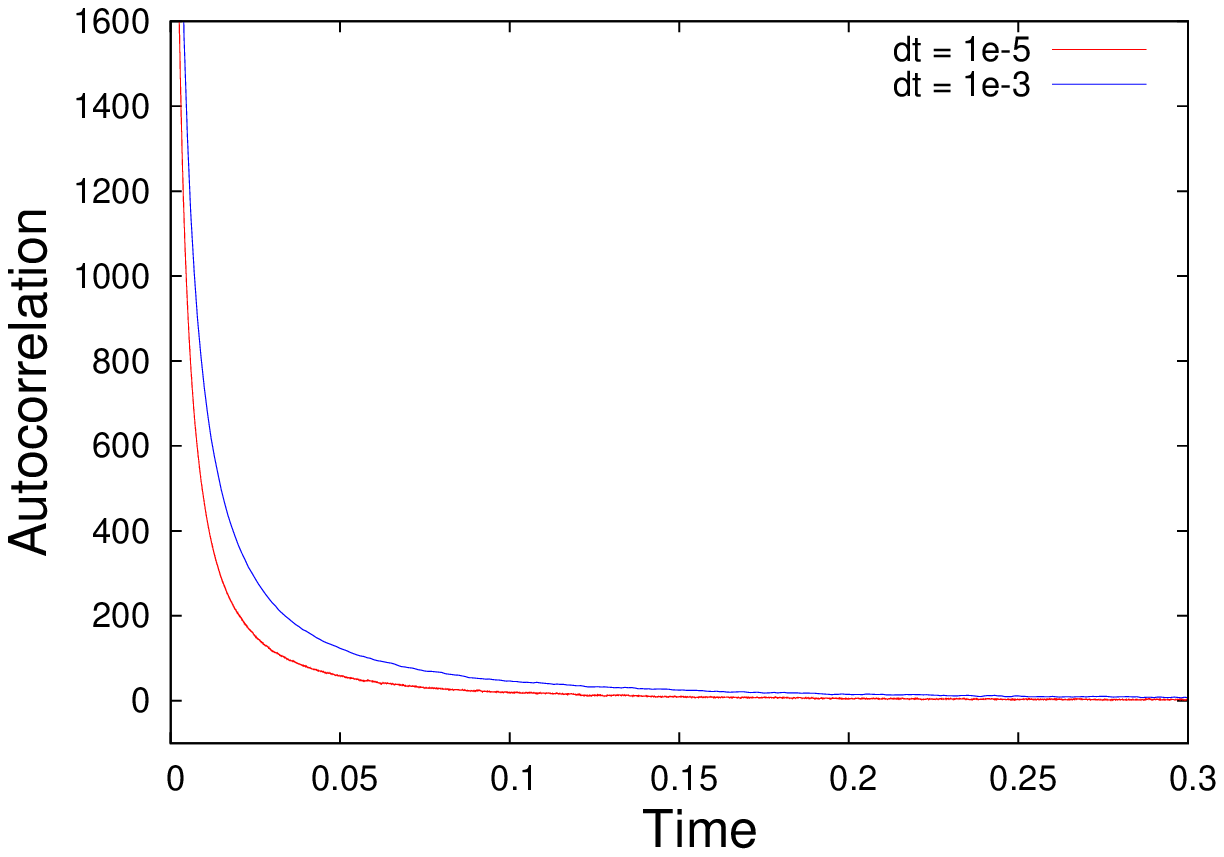}
\includegraphics[width=7cm]{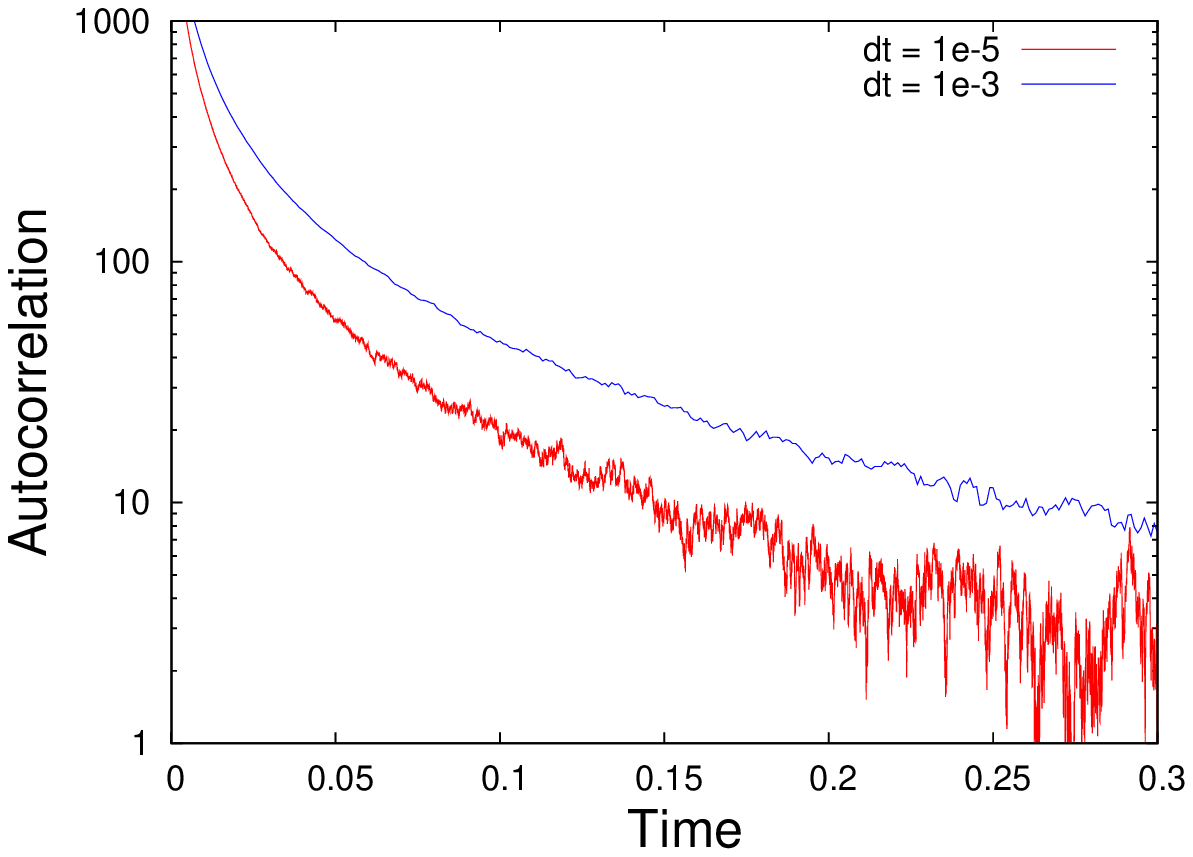}
\end{center}
\caption{\label{fig:correlation_LJ} Plot of the approximated correlation functions $\mathbb{E}(\nabla V(q_t)^T \nabla V(q_0))$. Left: standard view. Right: logarithmic scale on the ordinates.}
\end{figure}

\begin{figure}
\begin{center}
\includegraphics[width=7cm]{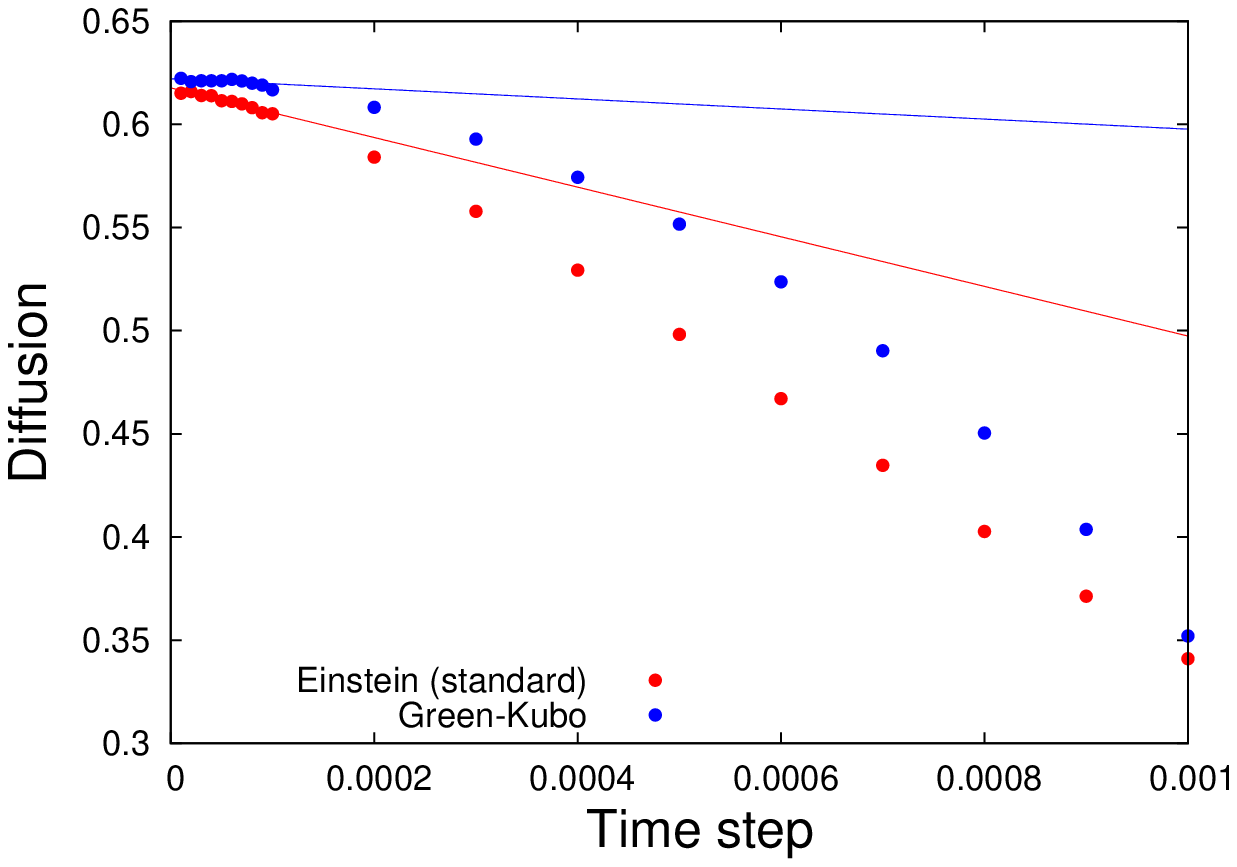}
\includegraphics[width=7cm]{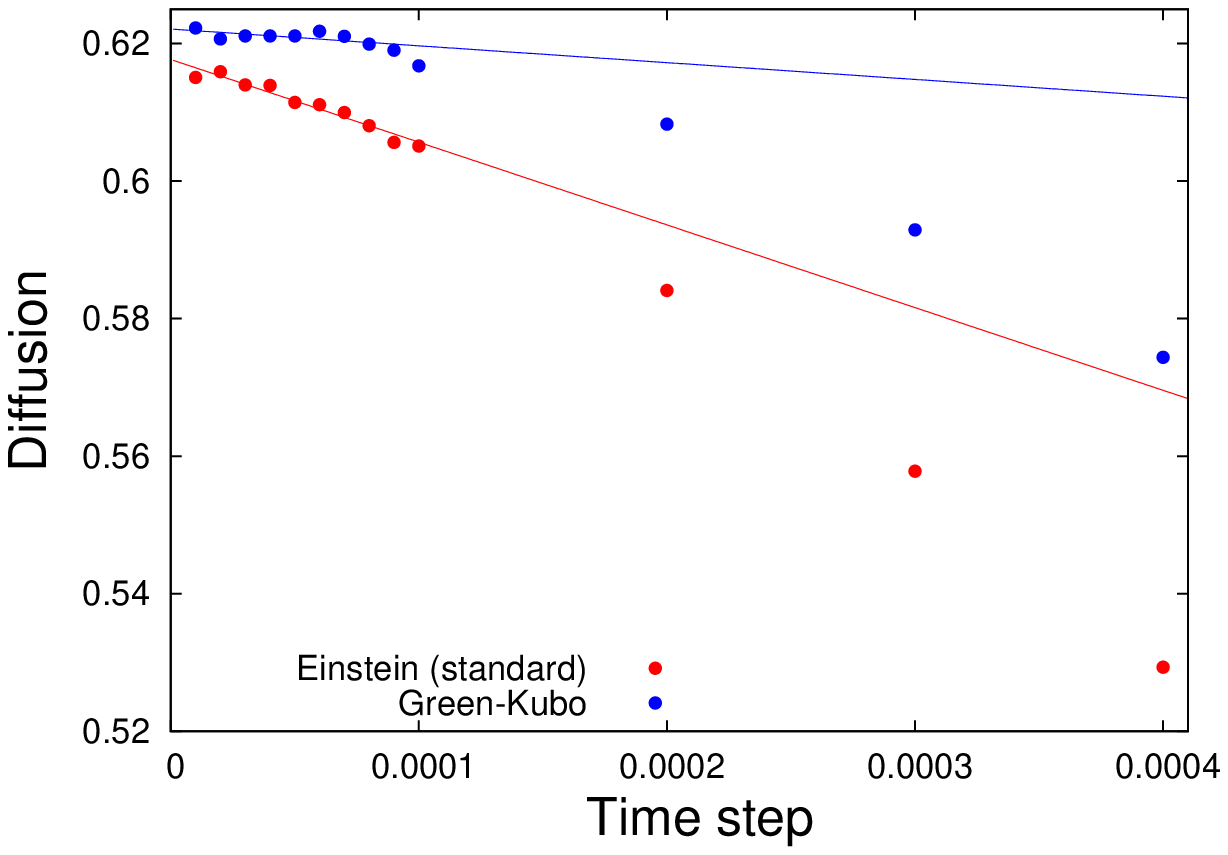}
\end{center}
\caption{\label{fig:error_diffusion_LJ} Diffusion constant as a function of the time step $\dt$ for the solvated ion system at $\beta = 1$, with a zoom on the smaller time steps on the right picture. Affine fits $\mathcal{D}_{\Delta} = \mathcal{D} + \Delta t \, \widetilde{D}_1$ consistent with~\eqref{eq:error_bounds_GK} and~\eqref{eq:error_bounds_Einstein} are in both cases represented by solid lines.}
\end{figure}

\section{Possible work tracks to reduce the error on the estimation of the self-diffusion}
\label{sec:tracks}

Both the Green-Kubo and the  Einstein approaches lead to discretization errors of order~$\dt$, as proved theoretically and verified numerically. A natural question is how to reduce this error. In the chemistry literature, it was proposed in~\cite{HB98} to renormalize the time in Einstein's method by replacing the simulation time~$n\dt$ appearing in the denominator of~\eqref{eq:def_D_Einstein_dt} by some effective time $\theta_\dt n \dt$, where $\theta_\dt$ is the average acceptance rate of the Metropolis algorithm for a given time step~$\dt$. However, since the average acceptance rate is of order $1-C\dt^{3/2}$ for small time steps (see~\eqref{eq:estimate_R_lemma}), such a correction cannot possibly cancel out the $\dt$-error in the diffusion coefficient.

A more promising work track, which we started working on at the end of our stay at CIRM, is to modify the proposed move (instead of simply considering the Euler-Maruyama scheme~\eqref{eq:EM}), and possibly the invariant measure as well, in order to increase the weak and strong orders of the associated Metropolized scheme. Some steps in this direction have already been pursued in~\cite{BDE13}. In fact, the proofs of Theorem~\ref{thm:GK} and~\ref{thm:Einstein} show that, in order to gain accuracy on the computation of transport coefficients, it is sufficient to find a numerical scheme such that
\[
P_\dt \psi = \left( \Id + \dt \, \cL + \frac{\dt^2}{2} \, \cL^2 \right) \psi + \dt^{5/2} r_{\psi,\dt}. 
\]
A key element to obtain such equalities is to decrease the rejection rate for small time steps. 

\section{Proof of the results}
\label{sec:proofs}

In all the proofs, the constants $C > 0$ and the critical time steps $\dt^*$ may change from line to line. Upon changing~$V$ into~$\beta V$, we may also assume that $\beta = 1$. 

\subsection{Proof of Theorem~\ref{thm:cv_faible_processus}}
\label{sec:cv_faible_processus}

The idea is to rewrite the part of the additive functional involving $- \nabla V(q_t)$ as an approximate martingale. To this end, we introduce the solutions $\Phi_0 = (\Phi_{0,1},..,\Phi_{0,dN})$ of the following Poisson equations 
\begin{equation}
\label{eq:Poisson_continu}
\mathcal{L}\Phi_{0,j} = -\nabla_{q_j} V.
\end{equation}
In view of the resolvent estimate~\eqref{eq:resolvent_estimate}, the functions $\Phi_{0,j}$ are well defined elements of $\widetilde{L}^2(\mu)$ since
\[
\int_\cM \nabla V \, d\mu = \frac1Z \int_\cM \nabla V \, \rme^{-\beta V} = -\frac1Z \int_\cM \nabla \left( \rme^{-\beta V} \right) = 0.
\]
In addition, by elliptic regularity, the functions $\Phi_{0,j}$ are smooth. By Ito's lemma, we therefore obtain
\[
d\Phi_{0,j}(q_t) = \cL \Phi_{0,j}(q_t)dt + \sqrt{2}\nabla \Phi_{0,j}(q_t) \cdot dW_t,
\]
so that the integrated displacement from the origin can be rewritten as
\begin{equation}
\label{eq:rewriting_increment_continuous}
Q_t - Q_0 = -\int_0^t \nabla V(q_s) \, ds  + \sqrt{2} \, W_t
= \Phi_0(q_t) - \Phi_0(q_0) +\sqrt{2}\int_0^t \Big(\mathrm{Id} - \nabla \Phi_0(q_s)\Big) dW_s.
\end{equation}
Since $\Phi_0$ is bounded, $(\Phi_0(X_t) - \Phi_0(X_0))/\sqrt{t}$ uniformly vanishes as $t$ goes to infinity. The long time behavior of the process $F^T (Q^\eps_t-Q^\eps_0)/\sqrt{t}$ (for a given direction $F \in \R^{dN})$ is therefore determined by the martingale 
\[
\mathscr{M}^{F,\eps}_t = \eps\sqrt{\frac2t} \int_0^{t/\eps^2} \Big(F - \nabla \left(F^T\Phi_0(q_s)\right)\Big) \cdot dW_s,
\]
whose quadratic variation is
\[
\left\langle \mathscr{M}^{F,\eps} \right\rangle_t = \frac{2\eps^2}{t}\int_0^{t/\eps^2} |F|^2 - 2 \nabla \left( F^T \Phi_0(q_s)\right) \cdot F + \left| \nabla \left(F^T \Phi_0(q_s)\right)\right|^2 \, ds.
\]
The ergodic properties of the diffusion process~$q_t$ allow to prove that
\begin{align}
F^T \mathscr{D} F = \lim_{t\to+\infty} \left\langle \mathscr{M}^{F,\eps} \right\rangle_t & = 2\int_\cM |F|^2 - 2 \nabla \left( F^T \Phi_0\right) \cdot F + \left| \nabla \left(F^T \Phi_0\right)\right|^2 \, d\mu \nonumber\\
& =2 \int_\cM |F|^2 - 2  \left( F^T \Phi_0\right) \left(F^T \nabla V\right) +  \cL\left[ \left(F^T \Phi_0\right)\right]\left(F^T \Phi_0\right)  \, d\mu \nonumber\\
& =2 \int_\cM |F|^2 -  \left( F^T \Phi_0\right) \left(F^T \nabla V\right) \, d\mu 
=2F^T \left( \Id + \int_\cM \cL^{-1} \left(\nabla V\right) \otimes \nabla V \, d\mu\right)F \label{eq:covariance}
\end{align}
where, to pass from the first to the second line, we have used an integration by parts to transform the second term in the integral and~\eqref{eq:quadratic_form} for the third one, while the two last equalities follow from the definition of~$\Phi_0$. At this stage, we note that $F^T \mathscr{D} F = 0$ implies that the integrand of the first equality vanishes almost everywhere, which, by a Cauchy-Schwarz inequality, in turn implies that $\nabla \left(F^T \Phi_0\right)$ is proportional to~$F$. This is however impossible since $F$ is not the gradient of a periodic function, and we therefore conclude that $\mathscr{D}$ is positive definite. In addition, the last expression shows that $\mathscr{D} \leq 2\Id$ since $-\cL^{-1}$ is a positive operator (replace $\varphi$ by $\cL^{-1} \phi$ in~\eqref{eq:spectral_gap}).

Since $-\mathcal{L}$ has a positive spectral gap on $L^2(\mu)$ (see~\eqref{eq:spectral_gap}), we can write the following operator equality on~$\widetilde{L}^2(\mu)$~:
\[
\mathcal{L}^{-1} = -\int_0^{+\infty} \rme^{t\mathcal{L}} \, dt.
\]
Therefore, for general functions $\psi,\varphi$ with vanishing average with respect to~$\mu$,
\begin{equation}
\label{eq:integrated_autocovariance}
\int_0^{+\infty} \mathbb{E}\Big[\psi(q_t)\,\varphi(q_0)\Big]dt = \int_\cM \left(-\mathcal{L}^{-1}\psi\right)\varphi \, d\mu.
\end{equation}
Combining this result with~\eqref{eq:covariance} leads to the expression~\eqref{eq:def_Einstein_continuous} of the diffusion matrix~$\mathscr{D}$. To prove the convergence of the processes, two arguments should be made precise (see~\cite{Olla} for an elementary account): 
\begin{enumerate}[(i)]
\item the convergence of the finite-dimensional laws, which can be obtained very simply here by considering the exponential martingales
\[
\exp\left[\mathrm{i}\theta \left(\mathscr{M}_t^{F,\eps}-\mathscr{M}_s^{F,\eps}\right) + \frac{\theta^2}{2}\Big(\left\langle \mathscr{M}^{F,\eps} \right\rangle_t-\left\langle \mathscr{M}^{F,\eps} \right\rangle_s\Big)\right],
\]
which are such that the conditional expectations converge to those of a Brownian motion as $\eps \to 0$:
\[
\lim_{\eps \to 0} \mathbb{E}\left( \left. \exp\left[\mathrm{i}\theta \left(\mathscr{M}_t^{F,\eps}-\mathscr{M}_s^{F,\eps}\right) + \frac{\theta^2}{2}\Big(\left\langle \mathscr{M}^{F,\eps} \right\rangle_t-\left\langle \mathscr{M}^{F,\eps} \right\rangle_s\Big)\right] \right| \, \mathcal{F}_{s/\eps^2} \right) = \exp\left(-\frac{\theta^2}{2} (t-s)F^T\mathscr{D}F\right),
\]
$\mathcal{F}_{s/\eps^2}$ denoting the filtration of events until the time $s/\eps^2$. Finite-dimensional laws are then obtained by a simple induction, as made precise in~\cite{BLP,Olla} for instance.
\item the tightness of the process, proved using Prohorov's criterion (see for instance~\cite{Billingsley99}): 
\[
\forall \alpha, \tau > 0, \qquad \lim_{\delta \to 0} \limsup_{\eps \to 0} \mathbb{P}\left( \sup_{\substack{ |t-s|<\delta \\ 0 \leq s < t \leq \tau}} \left| F^T \Big( Q_{t}^{\eps} - Q_s^{\eps} \Big)\right| \geq \alpha \right) = 0. 
\]
This criterion is satisfied in view of the tightness of the martingale $\mathscr{M}^{F,\eps}_t$, itself easily obtained using Doob's inequality (see~\cite{Olla}).
\end{enumerate}

\subsection{Proof of Lemma~\ref{lem:expansion_Pdt}}
\label{sec:proof_expansion_dt}

We first determine the magnitude of the acceptance probability in the Metropolis algorithm, which reads
\[
R_\dt(q^n,\widetilde{q}^{n+1}) = \min\left(1,\rme^{-\alpha_\dt(q^n,\widetilde{q}^{n+1})}\right), 
\]
with
\[
\begin{aligned}
\alpha_\dt(q,q') & = V(q') - V(q) + \frac{1}{4\dt}\left[ \left(q-q'+\dt \nabla V(q')\right)^2 - \left(q'-q+\dt \nabla V(q)\right)^2 \right] \\
& = V(q') - V(q) - \frac12 \left\langle q'-q,\nabla V(q')+\nabla V(q)\right\rangle + \frac{\dt}{4}\left( |\nabla V(q')|^2 - |\nabla V(q)|^2 \right).
\end{aligned}
\]
Using the following expansions (with integral remainders)
\begin{align}
V(q') - V(q) & = \left\langle \nabla V(q), q'-q\right\rangle + \frac12 (q'-q)^T \left[ \nabla^2 V(q) \right] (q'-q) + \frac16 D^3 V(q)\Big((q'-q)^{\otimes 3}\Big) \nonumber \\
& \ \ \ + \frac16 \int_0^1 (1-t)^3 D^4 V\big((1-t)q+tq'\big)\Big((q'-q)^{\otimes 4}\Big) \, dt, \label{eq:Taylor_3} \\
\nabla V(q') & = \nabla V(q) + \nabla^2 V(q) (q'-q) + \frac12 D^3 V(q)\Big((q'-q)^{\otimes 2}\Big) \nonumber \\
& \ \ \ + \frac12 \int_0^1 (1-t)^2 D^3 V\big((1-t)q+tq'\big)\Big((q'-q)^{\otimes 3}\Big) \, dt, \nonumber
\end{align}
a simple computation shows that
\begin{equation}
\label{eq:alpha_dt}
\alpha_\dt(q^n,\widetilde{q}^{n+1}) = \dt^{3/2} \xi\left(q^n,G^n\right) + \dt^2 \widetilde{\xi}_\dt(q^n,G^n), 
\end{equation}
where 
\begin{equation}
\label{eq:def_xi}
\xi(q,G) = -\frac{\sqrt{2}}{6} D^3 V(q) \Big( G^{\otimes 3} \Big) + \frac{\sqrt{2}}{2} \nabla V(q)^T \nabla^2 V(q) G,
\end{equation}
while there exists a constant $C > 0$ such that $\left| \widetilde{\xi}_\dt(q,G) \right| \leq C(1 + |G|^6)$ for any $0 \leq \dt \leq 1$. We next use the inequality
\[
x_+ - \frac{x_+^2}{2} \leq 1 - \min(1,\rme^{-x}) \leq x_+, \qquad x_+ = \max(0,x),
\]
obtained by distinguishing the cases $x \leq 0$ and $x \geq 0$. This shows that 
\begin{equation}
\label{eq:estimate_R}
R_\dt(q^n,\widetilde{q}^{n+1}) = 1 - \dt^{3/2} \xi_+\left(q^n,G^n\right) + \dt^2 \widehat{\xi}_\dt(q^n,G^n), 
\end{equation}
with $\xi_+(q,G) = \max(0,\xi(q,G))$ and where $\left|\widehat{\xi}_\dt(q,G)\right| \leq C(1+|G|^{12})$. The estimate on the average acceptance rate~\eqref{eq:estimate_R_lemma} is obtained by taking the expectation over all possible realizations of~$G^n$, with the definition
\[
\overline{\xi}(q) = \mathbb{E}_G \Big( \xi_+(q,G) \Big).
\]

To obtain the action of the operator~$A$, we start from the expression of the Metropolis transition operator (see for instance~\cite[Section~2.1.2]{LRS10})
\[
P_\dt\psi(q) = \int_\cM R_\dt(q,q') T_\dt(q,q') \psi(q') \, dq' + \left(1-\int_\cM R_\dt(q,q') T_\dt(q,q') \, dq'\right)\psi(q), 
\]
which can be reformulated as
\[
P_\dt\psi(q) - \psi(q) = \int_\cM R_\dt(q,q') T_\dt(q,q') \Big( \psi(q') - \psi(q) \Big)\, dq'. 
\]
We now write $q' = q - \dt \, \nabla V(q) + \sqrt{2\dt} \, g$, so that, using~\eqref{eq:estimate_R} to estimate the rejection rate,
\[
\begin{aligned}
& P_\dt\psi(q) - \psi(q) \\
& = \int_{\R^{dN}} R_\dt\left(q,q- \dt \, \nabla V(q) + \sqrt{2\dt} \, g\right) \Big( \psi\left(q- \dt \, \nabla V(q) + \sqrt{2\dt} \, g\right) - \psi(q) \Big)\, \frac{\rme^{-g^2/2}}{(2\pi)^{dN/2}} \, dg \\
& = \int_{\R^{dN}} \Big( \psi\left(q- \dt \, \nabla V(q) + \sqrt{2\dt} \, g\right) - \psi(q) \Big)\, \frac{\rme^{-g^2/2}}{(2\pi)^{dN/2}} \, dg \\
& + \int_{\R^{dN}} \left[ R_\dt\left(q,q- \dt \, \nabla V(q) + \sqrt{2\dt} \, g\right) -1 \right] \Big( \psi\left(q- \dt \, \nabla V(q) + \sqrt{2\dt} \, g\right) - \psi(q) \Big)\, \frac{\rme^{-g^2/2}}{(2\pi)^{dN/2}} \, dg \\
& = \dt\, (\cL\psi)(q) + \frac{\dt^2}{2} \left(\left[\cL^2 + D_1 + D_2 \right]\psi\right)(q) + \dt^{5/2} r_{\psi,\dt},
\end{aligned}
\]
where we have used for the first integral a Taylor expansion at fourth order similar to~\eqref{eq:Taylor_3} to obtain (see the computations in~\cite[Section~4.9]{LMS13})
\[
D_1\psi = 2\nabla^2 V : \nabla^2\psi + \nabla(\Delta V)\cdot \nabla\psi - \nabla V^T (\nabla^2 V) \nabla\psi\,;
\]
and a Taylor expansion at first order for the term involving the rejection rate to obtain
\[
D_2\psi = -\sqrt{2} \left(\int_{\R^{dN}} \xi_+(q,g) \, g \, \frac{\rme^{-g^2/2}}{(2\pi)^{dN/2}} \, dg\right)^T \nabla \psi.
\]
The remainder $r_{\psi,\dt}$ is uniformly bounded in $L^\infty(\cM)$ for $\dt$ sufficiently small. The conclusion follows by setting $A = (\cL^2 + D_1 + D_2)/2$.

Finally, the invariance of~$\mu$ by~$P_\dt$ implies
\[
\forall \dt > 0, \qquad \int_\cM P_\dt \psi \, d\mu = \int_\cM \psi \, d\mu.
\]
This equality, together with the expansion~\eqref{eq:expansion_P_dt} proves~\eqref{eq:vanishing_A}.

\subsection{Proof of Lemma \ref{lem:geom_ergod}}
\label{sec:proof_lem_geom_ergod}

The proof below is a simplification of the argument presented in~\cite{BH13}, made possible since we work on a compact state space. The idea of the proof is to compare the Metropolis dynamics to the continuous dynamics, using the standard, un-Metropolized Euler-Maruyama scheme as an intermediate. Alternatively, it would be possible to directly compare the Metropolized and un-Metropolized schemes, by proving as in~\cite[Section~4.2]{LMS13} that the standard, un-Metropolized Euler-Maruyama scheme is geometrically ergodic since it satisfies a minorization condition.

To prove the first part of the Lemma, it is enough to show that there exists $\rho < 1$ such that, for any $0 < \dt \leq \dt^*$ and any $q \in \T^{dN}$,
\[
\left\| P_\dt^{n\lfloor 1/\dt \rfloor}(q, \cdot) - \mu \right\|_{\rm TV} \leq C\rho^n.
\]
Since we work on a compact state space, it is enough to show by Harris' theorem (see the presentation in~\cite{HM11,BH13}) that there exist $\alpha \in (0,1)$ such that
\begin{equation} 
  \label{ineq_harris}
  \left\| P_\dt^{\lfloor 1/\dt \rfloor}(q, \cdot) - P_\dt^{\lfloor 1/\dt \rfloor}(q', \cdot)\right\|_{\rm TV} \leq 2(1 - \alpha),
\end{equation}
uniformly in $0 < \dt \leq \dt^*$ and $(q,q') \in \cM^2$. We now introduce the transition kernel $Q_t$ of the continuous dynamics~\eqref{eq:dynamics}, defined as
\[
Q_t \varphi(q) = \mathbb{E}(\varphi(q_t) \, | \, q_0 = q) = \left(\rme^{t \cL} \varphi\right)(q) = \int_\cM Q_t(q,q') \varphi(q') \, dq',
\]
and consider it at time $t=1$. The transition kernel is well defined and regular since the generator is elliptic. By the triangle inequality,
\begin{equation}
\label{eq:tri_ineq_transitions}
\begin{aligned}
\sup_{(q,q')\in\cM^2} \left\|P_\dt^{\lfloor 1/\dt \rfloor}(q, \cdot) - P_\dt^{\lfloor 1/\dt \rfloor}(q', \cdot)\right\|_{\rm TV} & 
\leq \sup_{(q,q')\in\cM^2} \left\| Q_1(q, \cdot) - Q_1(q', \cdot)\right\|_{\rm TV} \\
& \ \ \ + 2 \, \sup_{q\in\cM} \left\|P_\dt^{\lfloor 1/\dt \rfloor}(q, \cdot) - Q_1(q, \cdot)\right\|_{\rm TV}.
\end{aligned}
\end{equation}
From~\cite[Lemma 2.7]{BH13}, since we work on a compact space, we know that there exists $\varepsilon > 0$ such that 
\[
\sup_{(q,q')\in\cM^2} \left\| Q_1(q, \cdot) - Q_1(q', \cdot)\right\|_{\rm TV} \leq 2(1-\varepsilon).
\]
To control the second term in~\eqref{eq:tri_ineq_transitions}, we introduce the transition kernel of the standard, un-Metropolized Euler-Maruyama scheme~\eqref{eq:EM}, denoted by $\widetilde{P}_\dt$, and write
\[
\sup_{q\in\cM} \left\|P_\dt^{\lfloor 1/\dt \rfloor}(q, \cdot) - Q_1(q, \cdot)\right\|_{\rm TV} \leq \sup_{q\in\cM} \left\|\widetilde{P}_\dt^{\lfloor 1/\dt \rfloor}(q, \cdot) - Q_1(q, \cdot)\right\|_{\rm TV} + \sup_{q\in\cM} \left\|P_\dt^{\lfloor 1/\dt \rfloor}(q, \cdot) - \widetilde{P}_\dt^{\lfloor 1/\dt \rfloor}(q, \cdot)\right\|_{\rm TV}
\]
By~\cite[Lemma~4.2]{BH13}, the transition kernel of the standard, un-Metropolized dynamics is uniformly close to the transition kernel of the continuous dynamics when the state space is compact: There exists $C_{\rm EM}>0$ and $\dt^*>0$ such that, for any $0 < \dt \leq \dt^*$,
\[
\sup_{q\in\cM} \left\| \widetilde{P}_\dt^{\lfloor 1/\dt \rfloor}(q, \cdot) - Q_1(q, \cdot)\right\|_{\rm TV} \leq C_{\rm EM} \sqrt{\dt}.
\]

It therefore remains to control the distance between the transition rate of the Metropolized and un-Metropolized dynamics. It is at this stage that the argument of~\cite[Lemma~4.6]{BH13} can be simplified. As in the proof of this lemma, we use a coupling argument, and consider two chains $q^i$ and $\widetilde{q}^i$, corresponding respectively to the Metropolized dynamics and the standard un-Metropolized one, starting from the same initial condition $q^0 \in \cM$. The probability that $q^n \neq \widetilde{q}^n$ is bounded from above by the probability that $q^i \neq \widetilde{q}^i$ for some $1 \leq i \leq n$, \textit{i.e.} at least one rejection occurred along the discrete trajectory. Since the probability to reject the move from~$q^i$ to~$q^{i+1}$ is $1-R_\dt(q^i,q^{i+1})$, it holds
\[
\left\|P_\dt^n(q, \cdot) - \widetilde{P}_\dt^n(q, \cdot)\right\|_{\rm TV} \leq 2 \, \mathbb{P}\left[ \left. q^n \neq \widetilde{q}^n \, \right| \, q^0 = q \right] \leq 2 \sum_{i=1}^n \mathbb{E}\left[ \left. \left( 1-R_\dt(q^i,q^{i+1}) \right) \, \right| \, q^0 = q \right].
\] 
In view of~\eqref{eq:estimate_R}, there exists $C_{\rm reject}$ such that $\mathbb{E}\left[ \left. \left( 1-R_\dt(q^i,q^{i+1}) \right) \, \right| \, q^0 = q \right] \leq C_{\rm reject} \dt^{3/2}$ for $\dt$ sufficiently small. The sum from $i=1$ to $n = \lfloor 1/\dt \rfloor$ may then be estimated as
\[
\left\|P_\dt^{\lfloor 1/\dt \rfloor}(q, \cdot) - \widetilde{P}_\dt^{\lfloor 1/\dt \rfloor}(q, \cdot)\right\|_{\rm TV} \leq 2 \, C_{\rm reject} \sqrt{\dt}.
\]
The combination of all previous estimates finally gives~\eqref{ineq_harris} provided~$\dt$ is sufficiently small.

For the second part of the lemma, we note that the bounds on the powers of $P_\dt$ imply that the sum $\dps \sum_{n = 0}^{+\infty} P_\dt^n$ is absolutely convergent in $\dps \mathcal{B}\left(\widetilde{L}^\infty(\cM)\right)$, and it is then easily checked that $\Id - P_\dt$, considered as an operator on~$\widetilde{L}^\infty(\cM)$, is invertible and that
\begin{equation}
  \label{eq:inv_I_Pdt}
  (\Id - P_\dt)^{-1} = \sum_{n=0}^{+\infty} P_\dt^n.
\end{equation}
In fact,
\[
\left\| (\Id - P_\dt)^{-1}f \right\|_{L^\infty} 
= \left\| \sum_{n=0}^{+\infty} P_\dt^n f \right\|_{L^\infty} 
\leq \sum_{n=0}^{+\infty} \left\| P_\dt^n f \right\|_{L^\infty} 
\leq C \sum_{n=0}^{+\infty} \rme^{-\lambda n \dt}\|f\|_{L^\infty} 
\leq \frac{C}{1-\rme^{-\lambda \dt}}\|f\|_{L^\infty},
\]
from which the bound~\eqref{eq:bound_discrete_generator} immediately follows.

\subsection{Proof of Theorem~\ref{thm:GK}}
\label{sec:GK}

We follow the strategy of~\cite[Section~4.8]{LMS13}. The proof starts by noticing that the integrated correlation function can be written using $\cL^{-1}$, as made precise in~\eqref{eq:integrated_autocovariance}. The strategy of the proof is to write an approximation of $\cL^{-1}$ using the discrete evolution operator~$P_\dt$. 

In view of~\eqref{eq:inv_I_Pdt} and~\eqref{eq:expansion_P_dt}, and since $\widetilde{C}^\infty(\cM)$ is stable under $\cL^{-1}$, it holds
\[
\begin{aligned}
\left(-\mathcal{L}\right)^{-1}\psi &= \left(\dt \sum_{n=0}^{+\infty} P_\dt^n \right)\left(\frac{\Id - P_\dt}{\dt}\right)\left(-\mathcal{L}^{-1}\right)\psi \\
&= \left(\dt \sum_{n=0}^{+\infty} P_\dt^n \right) \left( \left(\Id + \dt A\mathcal{L}^{-1}\right)\psi + \dt^{3/2} r_{\cL^{-1}\psi,\dt} \right). 
\end{aligned}
\]
Since $\cL^{-1}\psi$ still is a smooth function (by elliptic regularity), the remainder $r_{\cL^{-1}\psi,\dt}$ is uniformly bounded by Lemma~\ref{lem:expansion_Pdt}. Note also that since $(\Id-P_\dt)\cL^{-1}\psi$ and $A\cL^{-1}\psi$ have vanishing averages with respect to~$\mu$, the remainder $r_{\cL^{-1}\psi,\dt}$ has a vanishing average with respect to~$\mu$. The above equality shows that
\[
\int_\cM \left(-\mathcal{L}^{-1}\psi\right)\varphi \, d\mu = \dt \sum_{n=0}^{+\infty} \int_\cM \left[ P_\dt^n \left( \widetilde{\psi}_\dt + \dt^{3/2} r_{\cL^{-1}\psi,\dt} \right) \right] \varphi \, d\mu,
\]
where the sum is convergent in view of~\eqref{eq:geom_ergod}. In conclusion,
\[
\int_\cM \left(-\mathcal{L}^{-1}\psi\right)\varphi \, d\mu = \dt \sum_{n=0}^{+\infty} \mathbb{E}_\dt\left(\widetilde{\psi}_\dt(q^n)\,\varphi(q^0) \right) + \dt^{3/2} \int_\cM \left[\left(\frac{\Id - P_\dt}{\dt}\right)^{-1} r_{\cL^{-1}\psi,\dt}\right] \varphi \, d\mu,
\]
which gives the result, in view of the boundedness of the operator $\dps \left(\frac{\Id - P_\dt}{\dt}\right)^{-1}$ on $\widetilde{L}^\infty(\cM)$ (given by Lemma~\ref{lem:geom_ergod}).

\subsection{Proof of Theorem~\ref{thm:Einstein}}
\label{sec:Einstein}

We start by highlighting the martingale part of the increments $\delta_\dt(q^n,G^n,U^n) = Q^{n+1}-Q^n$, similarly to the continuous case (compare~\eqref{eq:rewriting_increment_continuous})~:
\begin{equation}
\label{eq:decomposition_delta}
\delta_\dt(q^n,G^n,U^n) = \Big( \delta_\dt(q^n,G^n,U^n) - \mathbb{E}_{G,U}\big( \delta_\dt(q^n,G,U) \big) \Big) + \mathbb{E}_{G,U}\big( \delta_\dt(q^n,G,U) \big),
\end{equation}
where the expectation on the right-hand side is over the Gaussian random variable~$G$ and the uniform variable~$U$ (the configuration~$q^n$ being fixed). It will be useful to decompose the increment as
\[
\delta_\dt(q,G,U) = \sqrt{2\dt}\, G - \dt \, \nabla V(q) - \delta^{\rm reject}_\dt(q,G,U),
\]
with
\[
\delta^{\rm reject}_\dt(q,G,U) = \ind_{U > R_\dt\left(q,q-\dt \, \nabla V(q) + \sqrt{2\dt} \, G\right)}\Big( \sqrt{2\dt}\, G - \dt \, \nabla V(q) \Big).
\]
In view of Lemma~\ref{lem:expansion_delta} (which shows in particular that $\delta^{\rm reject}_\dt(q,G,U)$ can be thought of as being of order~$\dt^2$ by setting $p=1$ in~\eqref{eq:lp_bound_delta_reject}), the first term on the right-hand side of~\eqref{eq:decomposition_delta} is equal to $\sqrt{2\dt} \, G^n$ at dominant order in~$\dt$. This term therefore corresponds to the term $\sqrt{2}W_t$ in the decomposition~\eqref{eq:rewriting_increment_continuous}. The second term in the right-hand side of~\eqref{eq:decomposition_delta} is handled by introducing an appropriate Poisson equation (see~\eqref{eq:def_Phi_dt} below), which is the discrete analogue of~\eqref{eq:Poisson_continu}.

\begin{lemma}
\label{lem:expansion_delta}
The average increment has the following expansion in powers of~$\dt$:
\begin{equation}
  \label{eq:expansion_ovd_delta}
  \mathbb{E}_{G,U}[\delta_\dt\left(q,G,U\right)] = -\dt \nabla V(q) - \sqrt{2} \, \dt^{2} \, \mathbb{E}_{G}\left[ \xi_+(q,G) G \right] + \dt^{5/2} r_{\delta,\dt},
\end{equation}
where $r_{\delta,\dt}$ is uniformly bounded for~$\dt$ sufficiently small.
In addition, there exists a constant $K > 0$ such that $\left|\delta^{\rm reject}_\dt(q,G,U)\right| \leq K\left(1+\sqrt{\dt}|G|\right)$ and, for any $p > 0$,
\begin{equation}
\label{eq:lp_bound_delta_reject}
\mathbb{E}_{G,U}\left|\delta^{\rm reject}_\dt(q,G,U)\right|^p \leq C \, \dt^{(p+3)/2}.
\end{equation}
\end{lemma}

\begin{lemma}
\label{lem:expansion_Phi_dt}
There exists, for any $\dt > 0$, a unique function $\Phi_\dt = (\Phi_{\dt,1},\dots,\Phi_{\dt,dN}) \in \left(\widetilde{L}^\infty(\cM)\right)^{dN}$ such that
\begin{equation}
  \label{eq:def_Phi_dt}
  \left( P_\dt-\Id \right) \Phi_\dt(q) = \mathbb{E}_{G,U}\big[\delta_\dt\left(q,G,U\right)\big].
\end{equation}
Moreover, recalling the definition~\eqref{eq:Poisson_continu} of~$\Phi_0$,
\[
\Phi_\dt = \Phi_0 + \dt\,\widetilde{\Phi}^1 + \dt^{3/2} \, \Psi_\dt,
\]
where $\Psi_\dt$ is uniformly bounded for~$\dt$ sufficiently small, and $\widetilde{\Phi}^1$ is the unique solution of the Poisson equation
\begin{equation}
  \label{eq:Poisson_Phi_1}
  \cL \widetilde{\Phi}^1 = A\Phi_0 - \sqrt{2} \, \mathbb{E}_{G}\left[ \xi_+(q,G) G \right], \qquad \int_\cM \widetilde{\Phi}^1 \, d\mu = 0.
\end{equation}
\end{lemma}

In view of these results,
\[
\begin{aligned}
\delta_\dt(q^n,G^n,U^n) & = \Big( \delta_\dt(q^n,G^n,U^n) - \mathbb{E}_{G,U}\big( \delta_\dt(q^n,G,U) \big) \Big) + P_\dt\Phi_\dt(q^n) - \Phi_\dt(q^n) \\
& = M^n_\dt + \Phi_\dt(q^{n+1})-\Phi_\dt(q^n),
\end{aligned}
\]
with, upon rewriting $P_\dt\Phi_\dt(q^n)$ as $\mathbb{E}_{G,U}\left[ \Phi_\dt\big(q^n + \delta_\dt(q^n,G,U)\big) \right]$,
\begin{equation}
\label{eq:discrete_martingale}
\begin{aligned}
M^n_\dt & = \Big( \delta_\dt(q^n,G^n,U^n) - \mathbb{E}_{G,U}\big( \delta_\dt(q^n,G,U) \big) \Big) - \Big( \Phi_\dt(q^{n+1}) - P_\dt\Phi_\dt(q^n) \Big)\\
& = \Big( \delta_\dt(q^n,G^n,U^n) - \mathbb{E}_{G,U}\big( \delta_\dt(q^n,G,U) \big) \Big) \\
& \ \ \ - \Big( \Phi_\dt\big(q^n + \delta_\dt(q^n,G^n,U^n)\big)- \mathbb{E}_{G,U}\left[ \Phi_\dt\big(q^n + \delta_\dt(q^n,G,U)\big) \right]\Big).
\end{aligned}
\end{equation}
The interest of this rewriting is to highlight the fact that $M^n_\dt$ can be fully understood in terms of the increments~$\delta_\dt(q^n,G^n,U^n)$ (in order to use Lemma~\ref{lem:cross_correlation} below). Note that $(M^n_\dt)_{n \geq 0}$ are stationary, independent martingale increments when $q^0~\sim \mu$ (since in this case $q^n \sim \mu$ for all $n \geq 0$). This shows that
\[
Q^n - Q^0 = \Phi_\dt(q^n)-\Phi_\dt(q^0) + \sum_{m=0}^{n-1} M^k_\dt.
\]
Since $\Phi_\dt$ is uniformly bounded as $\dt \to 0$, we obtain
\[
\mathscr{D}_\dt^{\rm Einstein} = \mathbb{E}\left(\frac{M^0_\dt}{\sqrt{\dt}} \otimes \frac{M^0_\dt}{\sqrt{\dt}}\right).
\]
We now expand $M^0_\dt$ in powers of~$\dt$. By Lemma~\ref{lem:expansion_Phi_dt}, it is possible to replace the function $\Phi_\dt$ in the second term on the right-hand side of~\eqref{eq:discrete_martingale} by $\Phi_0 + \dt \widetilde{\Phi}^1$, up to a remainder of order~$\dt^{3/2}$. We next use the following lemma to compute the cross correlation between the various functions of~$\dt$ appearing in $M^0_\dt \otimes M^0_\dt$.

\begin{lemma}
\label{lem:cross_correlation}
For any smooth functions $f,g$, growing at most polynomially,
\[
\frac{1}{\dt} \mathbb{E}_{G,U}\left[ \Big(f(\delta_\dt(q,G,U)) - \overline{f}(q) \Big)\Big(g(\delta_\dt(q,G,U)) - \overline{g}(q) \Big) \right] 
= 2 \nabla f(0)^T \nabla g(0) + \dt\,r_{f,g,\dt},
\]
with $\overline{f}(q) = \mathbb{E}_{G,U}\left[ f(\delta_\dt(q,G,U)) \right]$, and where the remainder $r_{f,g,\dt}$ is uniformly bounded for~$\dt$ sufficiently small.
\end{lemma}

The conclusion then follows by applying this result with the functions $f,g$ replaced by $x \mapsto F^Tx$, $x \mapsto F^T \Phi_0(q+x)$ and $x \mapsto F^T \widetilde{\Phi}^1(q+x)$ for a given test direction~$F$. Indeed,
\[
F^T \mathscr{D}_\dt^{\rm Einstein} F = 2 \int_\cM |F|^2 - 2 \nabla \left(F^T \Phi_0\right) \cdot F + \left|\nabla \left(F^T \Phi_0\right)\right|^2 \, d\mu + \dt \, r_{F,\dt},
\]
with $|r_{F,\dt}|/|F|^2$ uniformly bounded as $\dt \to 0$. Manipulations similar to the ones leading to~\eqref{eq:covariance} finally give the claimed result.

\begin{remark}
\label{rem:more_work}
In order to characterize the leading order term in the error and prove that the subleading order term indeed is of order~$\dt^{3/2}$, it would be necessary to compute correlation terms involving components of the remainder term~$\Psi_\dt$. This is not possible as such because the regularity of~$\Psi_\dt$ is not established, and obtaining regularity result from~\eqref{eq:def_Phi_dt} is difficult. An expansion of $\Phi_\dt$ up to $\dt^2$ terms (as $\Phi_0 + \dt \widetilde{\Phi}^1 + \dt^{3/2} \widetilde{\Phi}^{3/2} + \dt^2 \widetilde{\Psi}_\dt$) is therefore needed in order not have to treat correlations involving the remainder~$\Psi_\dt$. This, in turn, would require an expansion of~$P_\dt$ up to remainders of order~$\dt^3$, instead of~$\dt^{5/2}$ as in~\eqref{eq:expansion_P_dt}. Although this does not pose any problem in principle, we chose not to follow this path in order to keep the arguments as simple as possible.
\end{remark}

We conclude this section with the proofs of the technical results quoted above.

\begin{proof}[Proof of Lemma~\ref{lem:expansion_delta}]
We first write an expansion of $\mathbb{E}_{G,U}(\delta_\dt(q,G,U))$ in fractional powers of~$\dt$:
\[
\begin{aligned}
\mathbb{E}_{G,U}[\delta_\dt\left(q,G,U\right)] + \dt \, \nabla V(q) & = \mathbb{E}_{G,U}\left[\delta_\dt\left(q,G,U\right) + \dt \, \nabla V(q) -\sqrt{2\dt} \, G\right]\\
& = \mathbb{E}_{G,U} \left[ \left( \ind_{U \leq R_\dt\left(q,q-\dt \, \nabla V(q) + \sqrt{2\dt} \, G\right)} -1\right)\left(-\dt \, \nabla V(q) + \sqrt{2\dt} \, G\right) \right] \\
& = \mathbb{E}_{G} \left[ \left( R_\dt\left(q,q-\dt \, \nabla V(q) + \sqrt{2\dt} \, G\right) -1\right)\left(-\dt \, \nabla V(q) + \sqrt{2\dt} \, G\right) \right].
\end{aligned}
\]
In view of~\eqref{eq:estimate_R}, it holds $\mathbb{E}_{G,U}[\delta_\dt\left(q,G,U\right)] + \dt \, \nabla V(q) = -\sqrt{2} \dt^2 \, \mathbb{E}_{G}\left[ \xi_+(q,G) G \right] + \dt^{5/2} r_{\delta,\dt}$, which gives~\eqref{eq:expansion_ovd_delta}. 

The bound~\eqref{eq:lp_bound_delta_reject} is a straightforward consequence of the equality
\[
\mathbb{E}_{G,U}\left|\delta^{\rm reject}_\dt(q,G,U)\right|^p = \mathbb{E}_{G}\left[\left( 1-R_\dt\left(q,q-\dt \, \nabla V(q) + \sqrt{2\dt} \, G\right)\right)\left| \sqrt{2\dt}\, G - \dt \, \nabla V(q) \right|^p \right],
\]
while $\left|\delta^{\rm reject}_\dt(q,G,U)\right| \leq \left|\sqrt{2\dt} \, G - \dt\, \nabla V(q)\right|$ immediately gives $\left|\delta^{\rm reject}_\dt(q,G,U)\right| \leq K\left(1+\sqrt{\dt}|G|\right)$.
\end{proof}

\begin{proof}[Proof of Lemma~\ref{lem:expansion_Phi_dt}]
We introduce the normalized average increment, defined as the following periodic function~:
\[
\overline{\delta}_\dt(q) = \frac{\mathbb{E}_{G,U}[\delta_\dt\left(q,G,U\right)]}{\dt}.
\]
Lemma~\ref{lem:geom_ergod} shows that $\Phi_\dt$ is well defined provided the periodic function $\overline{\delta}_\dt$ has a vanishing average with respect to~$\mu$. To prove this statement, we start from
\[
\mathbb{E}_{G,U}[\delta_\dt\left(q,G,U\right)] = \int_{\R^{dN}} R_\dt(q,q')T_\dt(q,q')(q'-q) \, dq'.
\]
It is easily seen that
\[
\begin{aligned}
& \int_\cM \mathbb{E}_{G,U}[\delta_\dt\left(q,G,U\right)] \, \mu(dq) = \int_\cM \int_{\R^{dN}} R_\dt(q,q')T_\dt(q,q')(q'-q) \, dq' \, \mu(dq) \\
& \qquad = \sum_{n \in \mathbb{Z}^{dN}} \int_\cM \int_\cM R_\dt(q,q'+nL)T_\dt(q,q'+nL)(q'-q+nL) \, dq' \, \mu(dq) = \sum_{n \in \mathbb{Z}^{dN}} I_n,
\end{aligned}
\]
with
\[
I_n = \int_\cM \int_\cM \min\Big( \mu(dq')T_\dt(q'+nL,q), \mu(dq)T_\dt(q,q'+nL)\Big) (q'-q+nL) \, dq' \, dq.
\]
Since $q$ and $q'$ play symmetric roles, $I_0$ vanishes. For $n \geq 0$, we obtain, by first exchanging the names of the dummy variables $q$ and $q'$ and then using $T_\dt(q+nL,q') = T_\dt(q,q'-nL)$ as well as the invariance of $\mu$ by translations of the periodic cell,
\[
\begin{aligned}
& I_n = \int_\cM \int_\cM \min\Big( \mu(dq')T_\dt(q'+nL,q), \mu(dq)T_\dt(q,q'+nL)\Big) (q'-q+nL) \, dq' \, dq \\
& \qquad \qquad  = \int_\cM \int_\cM \min\Big( \mu(dq)T_\dt(q+nL,q'), \mu(dq')T_\dt(q',q+nL)\Big) (q-q'+nL) \, dq \, dq' \\
& \qquad \qquad  = \int_\cM \int_\cM \min\Big( \mu(dq)T_\dt(q,q'-nL), \mu(dq')T_\dt(q'-nL,q)\Big) \Big(q-(q'-nL)\Big) \, dq \, dq' \\
& \qquad \qquad  = -\int_\cM \int_\cM \min\Big( \mu(dq')T_\dt(q'-nL,q),\mu(dq)T_\dt(q,q'-nL)\Big) (q'-q-nL) \, dq \, dq'.
\end{aligned}
\]
Therefore, $I_n = -I_{-n}$, which allows to conclude that
\begin{equation}
\label{eq:vanishing_avg_ovd_delta}
\forall \dt > 0, \qquad \int_\cM \overline{\delta}_\dt(q) \, \mu(dq) = 0.
\end{equation}

To obtain an expansion of $\Phi_\dt$ in terms of fractional powers of~$\dt$, we rely on~\eqref{eq:expansion_ovd_delta}, which implies that
\begin{equation}
\label{eq:expansion_overline_delta}
\overline{\delta}_\dt(q) = -\nabla V(q) - \sqrt{2} \, \dt \, \mathbb{E}_{G}\left[ \xi_+(q,G) G \right] + \dt^{3/2} r_{\delta,\dt}.
\end{equation}
The function $\widetilde{\Phi}^1$ introduced in~\eqref{eq:Poisson_Phi_1} is indeed well defined since $Af$ has a vanishing average with respect to~$\mu$ for any smooth function~$f$ (see~\eqref{eq:vanishing_A}), while the condition~\eqref{eq:vanishing_avg_ovd_delta}, together with the expansion~\eqref{eq:expansion_overline_delta}, shows that the average of $\mathbb{E}_{G}\left[ \xi_+(q,G) G \right]$ with respect to~$\mu$ also vanishes. Now, consider the following difference, relying on the expansion~\eqref{eq:expansion_P_dt} and the definition of~$\Phi_0$~:
\[
\begin{aligned}
-\left(\frac{\mathrm{Id}-P_\dt}{\dt}\right) \left( \Phi_\dt - \Phi_0 - \dt \, \widetilde{\Phi}^1 \right) & = \overline{\delta}_\dt(q) - \Big(\cL + \dt \, A\Big) \left( \Phi_0 + \dt \, \widetilde{\Phi}^1 \right) + \dt^{3/2} \, r_\dt \\
& = -\dt \left( A\Phi_0 + \sqrt{2} \, \mathbb{E}_{G}\left[ \xi_+(q,G) G \right] + \cL \widetilde{\Phi}^1 \right) + \dt^{3/2} \, \widetilde{r}_\dt.
\end{aligned}
\]
The first term on the right-hand side of the last equality vanishes by definition of $\widetilde{\Phi}^1$. The remainder $\widetilde{r}_\dt$ has a vanishing average with respect to~$\mu$ since it belongs to $\mathrm{Ran}(\Id-P_\dt)$. Lemma~\ref{lem:geom_ergod} then shows that there exists a constant $C > 0$ and $\dt^* > 0$ such that, for any $0 < \dt \leq \dt^*$, 
\[
\left\| \Phi_\dt - \Phi_0 - \dt \, \widetilde{\Phi}^1 \right\|_{L^\infty(\cM)} \leq C \, \dt^{3/2}.
\]
This gives the result upon defining $\Psi_\dt = \dt^{-3/2}\left(\Phi_\dt - \Phi_0 - \dt \, \widetilde{\Phi}^1\right)$.
\end{proof}

\begin{proof}[Proof of Lemma~\ref{lem:cross_correlation}]
A Taylor expansion with integral remainder gives
\[
\begin{aligned}
f(\delta_\dt(q,G,U)) & = f(0) + \nabla f(0)^T \delta_\dt(q,G,U) + \frac12 \delta_\dt(q,G,U)^T \nabla^2 f(0) \delta_\dt(q,G,U) \\
& \ \ \ + \frac12 \int_0^1 (1-\theta)^2 D^3f(\theta \delta_\dt(q,G,U)) \Big(\delta_\dt(q,G,U)^{\otimes 3}\Big)d\theta.
\end{aligned}
\]
A simple computation using~\eqref{eq:expansion_ovd_delta} and~\eqref{eq:lp_bound_delta_reject} shows that
\[
\overline{f}(q) = f(0) + \dt \left( -\nabla V(q)^T \nabla f(0) + \Delta f(0) \right) + \dt^2 r_{f,\dt},
\]
where $r_{f,\dt}$ is uniformly bounded for $\dt$ sufficiently small. Therefore,
\[
\begin{aligned}
& f(\delta_\dt(q,G,U)) - \overline{f}(q) \\
& = \nabla f(0)^T \Big( \delta_\dt(q,G,U) + \dt\, \nabla V(q)\Big) + \frac12 \nabla^2 f(0) : \Big( \delta_\dt(q,G,U) \otimes \delta_\dt(q,G,U) - 2\dt\,\Id\Big)\\
& \quad + \frac12 \int_0^1 (1-\theta)^2 D^3f(\theta \delta_\dt(q,G,U)) \Big(\delta_\dt(q,G,U)^{\otimes 3}\Big)d\theta - \dt^2 r_{f,\dt} \\
& = \sqrt{2\dt} \, \nabla f(0)^T G + \dt \, \nabla^2 f(0) : \Big( G \otimes G - \Id\Big) - \frac{\dt^{3/2}}{\sqrt{2}} \nabla^2 f(0) : \Big(\nabla V(q) \otimes G + G \otimes \nabla V(q) \Big)  \\
& \quad + \nabla f(0)^T \delta^{\rm reject}_\dt(q,G,U) \\
& \quad  + \frac12 \nabla^2 f(0) : \Big( \delta^{\rm reject}_\dt(q,G,U) \otimes \delta_\dt(q,G,U) + \delta_\dt(q,G,U) \otimes \delta^{\rm reject}_\dt(q,G,U) + \delta^{\rm reject}_\dt(q,G,U) \otimes \delta^{\rm reject}_\dt(q,G,U)\Big)\\
& \quad + \frac12 \int_0^1 (1-\theta)^2 D^3f(\theta \delta_\dt(q,G,U)) \Big(\delta_\dt(q,G,U)^{\otimes 3}\Big)d\theta - \dt^2 r_{f,\dt}.
\end{aligned}
\]
A simple computation finally shows that
\[
\frac{1}{\dt} \mathbb{E}_{G,U}\left[ \Big(f(\delta_\dt(q,G,U)) - \overline{f}(q) \Big)\Big(g(\delta_\dt(q,G,U)) - \overline{g}(q) \Big) \right] 
= 2 \, \mathbb{E}_{G}\left[\left( \nabla f(0)^T G \right)\left(\nabla g(0)^T G\right)\right] + \dt\, r_{f,g,\dt},
\]
where the remainder $r_{f,g,\dt}$ is uniformly bounded.
\end{proof}


\end{document}